\documentclass[preprint,10pt]{article}
\usepackage{amssymb}
\usepackage{amsthm}
\usepackage{amsmath}
\usepackage{url}
\usepackage{mathrsfs}
\newtheorem{thm}{Theorem}[section]
\newtheorem{rem}{Remark}[section]
\newtheorem{defn}{Definition}[section]
\newtheorem{lemm}{Lemma}[section]
\newtheorem{prop}{Proposition}[section]
\newcommand{\de}{\delta}
\newcommand{\la}{\lambda}
\newcommand{\ep}{\varepsilon}
\newcommand{\ph}{\varphi}
\newcommand{\Om}{\Omega}

\newcommand{\si}{\sigma}
\newcommand{\al}{\alpha}
\newcommand{\pa}{\partial}

\newcommand{\mN}{\mathbb{N}}
\newcommand{\mR}{\mathbb{R}}
\newcommand{\mE}{\mathbb{E}}
\newcommand{\mT}{\mathbb{T}}
\newcommand{\nn}{\nonumber}
\newcommand{\hs}{\hspace}

\usepackage{authblk}
\title{A time-splitting approach to quasilinear Degenerate Parabolic Stochastic Partial Differential Equations}
\author[1]{Kazuo Kobayasi}
\author[*,2]{Dai Noboriguchi}
\affil[1,2]{Department of Mathematics, Education and Integrated Arts and Sciences, Waseda University, 1-6-1 Nishi-Waseda, Shinjuku-ku, Tokyo, 169-8050, Japan}
\date{revised version: April 21, 2016 \\ \mbox{} \\ Accepted for publication in Differential and Integral Equations}
\begin{document}
\maketitle
\begin{abstract}
In this paper, we discuss the Cauchy problem for a degenerate parabolic-hyperbolic equation with a multiplicative noise. We focus on the existence of a solution. Using nondegenerate smooth approximations, Debussche, Hofmanov\'a and Vovelle \cite{DeHoVo} proved the existence of a kinetic solution. On the other hand, we propose to construct a sequence of approximations by applying a time splitting method and prove that this converges strongly in $L^1$ to a kinetic solution. This method will somewhat give us not only a simpler and more direct argument but an improvement over the existence result.
\end{abstract}
\footnote[0]{\hs{-5mm} $^*$Corresponding author. \\ \ {\it Email addresses:} kzokoba@waseda.jp (K. Kobayasi), 588243-dai@fuji.waseda.jp (D. Noboriguchi). \\ \ {\it The current address:} noboriguchi@kushiro-ct.ac.jp (D. Noboriguchi).}

\section{Introduction} 

In this paper we study a quasilinear degenerate parabolic stochastic partial differential equation of the following type
\begin{gather}
du + {\rm div} (B (u)) dt = {\rm div} (A(u) \nabla u) dt + \Phi (u) dW(t) \hspace{5mm} {\rm in} \hspace{3mm} \mT^d \times (0,T), \label{dp1}
\end{gather}with the initial condition
\begin{gather}
u (\cdot,0) = u_0 (\cdot) \hspace{5mm} {\rm in} \hspace{3mm} \mT^d, \label{dp2}
\end{gather}where $\mT^d$ is the $d$-dimensional torus and $W$ is a cylindrical Wiener process defined on a stochastic basis $(\Omega, \mathscr{F}, ( \mathscr{F}_t ), P)$. More precisely, $(\mathscr{F}_t)$ is a complete right-continuous filtration and $W (t) = \sum_{k=1}^\infty \beta_k (t) e_k$ with $(\beta_k)_{k \geq 1}$ being mutually independent real-valued standard Wiener processes relative to $(\mathscr{F}_t)$ and $(e_k)_{k \geq 1}$ a complete orthonormal system in a separable Hilbert space $H$ (cf. \cite{DaZa} for example).


In the deterministic and completely degenerate case of $g = 0$ and $A=0$, the problem has been extensively studied by many authors \cite{ImVo}, \cite{Kr}, \cite{LiPeTa}, \cite{MaNeRoRu}, \cite{Pe}. It is well known that a smooth solution is constant along characteristic curves, which can intersect each other and shocks can occur. Consequently classical solutions do not exist in general on the whole interval $[0,T]$ and some weak solution must be considered. However, it was shown that weak or distributional solutions lack uniqueness and therefore additional conditions need to ensure uniqueness of weak solutions.


In the deterministic case that $\Phi = 0$ the concept of entropy solution was introduced by Kru\v{z}kov \cite{Kr} for a first-order scalar conservation law (i.e., $A=0$) and then extended by Carrillo \cite{Ca} for a quasilinear degenerate parabolic equation. On the other hand, the concept of kinetic solution was introduced by Lions, Perthame and Tadmor \cite{LiPeTa} for a deterministic first-order scalar conservation law and futher studied in \cite{ImVo}, \cite{Pe}. The relationship between entropy solution and kinetic one will be found in \cite{Pe}, for example. These solutions have been extended for quasilinear degenerate parabolic equations, see \cite{Ca}, \cite{ChPe}, \cite{Ko}, \cite{KoOh}, \cite{MiVo}, for example.


\indent To perturb a stochastic term is natural for applications, which appears in wide variety of fields as physics, engineering and others. There are many papers concerning first-order scalar conservation laws with stochastic forcing. The uniqueness and the existence of entropy solution to the equations with an additive noise ($\Phi$ independent of $u$) have been studied in \cite{Ki}, with multiplicative noise in \cite{BaVaWi}, \cite{ChDiKa}, \cite{FeNu}. Among other things, Debussche and Vovelle \cite{DeVo}, \cite{DeVo2} established the uniqueness and the existence of kinetic solutions to the equation with multiplicative noise by using a kinetic formulation which keeps track of the dissipation of noise by solutions. Also see \cite{BaVaWi}, \cite{KoNo} in the case of initial-boundary value problem.


\indent There are a few papers concerning the Cauchy problem (\ref{dp1}), (\ref{dp2}) for general degenerate parabolic stochastic equations. First, Hofmanov\'a \cite{Ho} proved the uniqueness and the existence of a kinetic solution in the semilinear case: in (\ref{dp1}) the diffusion matrix $A(u)$ is replaced by a matrix $A(x)$ not depending on $u$ but possibly depending on $x$. Then, Debussche, Hofmanov\'a and Vovelle \cite{DeHoVo} studied the quasilinear case of type (\ref{dp1}), in a similar framework as in \cite{DeVo}, \cite{DeVo2}. We can also find the paper \cite{BaVaWi2} on degenerate SPDEs.


\indent We recall the proofs of the existence of kinetic solutions in the previous papers \cite{DeHoVo}, \cite{DeVo}, \cite{DeVo2}. In the case of hyperbolic conservation laws \cite{DeVo}, \cite{DeVo2} the authors introduced a notion of generalized kinetic solution and obtained a comparison theorem for any two generalized kinetic solutions, which assures us that any generalized kinetic solution is actually necessarily a kinetic one. This reduction result simplifies the proof of existence since only weak convergence of approximate viscous solutions is needed.


\indent On the other hand, in the case of degenerate parabolic equations the comparison result for generalized kinetic solutions cannot be obtained. In \cite{DeHoVo} such a result is obtained only for kinetic solutions and therefore strong convergence of approximate solutions is needed for the proof of existence. Towards this end a compactness argument is employed: uniform estimates together with the Prokhorov theorem and the Skorokhod theorem yield the strong convergence of the subsequence of approximate solutions on another probability space and the limit becomes a martingale kinetic solution. The existence of kinetic solution is then obtained by virtue of Gy\"ongy and Krylov characterization of convergence in probability (\cite{GyKr}). Thus, in order to obtain the existence of kinetic solutions the authors \cite{DeHoVo} approximated the equation (\ref{dp1}) by nondegenerate equations having smooth coefficients one after another. However it would seem that it is not easy to show the existence of solutions to such approximate equations.


\indent Our purpose of the paper is to give another proof of existence of kinetic solutions to the Cauchy problem (\ref{dp1}), (\ref{dp2}) and to extend slightly the existence result of \cite{DeHoVo} by a time-splitting method. We will employ the notion of (stochastic) kinetic solution used in \cite{DeHoVo}; moreover, we will frequently use the arguments developed in \cite{DeVo}, \cite{DeVo2}. As was mentioned in the paper of Holden and Risebro \cite{HoRi}, our method is to split the effect of deterministic degenerate parabolic equation and the stochastic source term in order to construct approximate kinetic solutions. We refer to the paper of Bauzet \cite{Ba} in which the author applied time-splitting method to a stochastic scalar conservation law in the framework of the space of integrable functions with bounded variation.


To describe our strategy in more detail, let $R(t,s) v_s$ denote the solution of the purely stochastic equation (\ref{R}) below with the initial $v_s$ at $t=s$, and let $S(t-s) u_s$ denote the solution of the deterministic degenerate parabolic equation (\ref{S}) below with the initial $u_s$ at $t=s$. Given $\ep > 0$ let $0 = t^\ep_0 < t^\ep_1 < \cdots < t^\ep_{N_\ep} = T$ be a partition of the interval $[0,T]$ such that the mesh size tends to $0$ as $\ep \to 0$. Consider the type of Lie-Trotter's product formula:
\begin{gather*}
 v^\ep (t) = R (t,t^\ep_n) \prod_{k = 1}^n \big[ S(t^\ep_{k} - t^\ep_{k-1}) R(t^\ep_{k}, t^\ep_{k-1}) \big] u_0
\end{gather*}for $t \in [0,T)$ where $n$ is the integer such that $t \in [t^\ep_n, t^\ep_{n+1})$. Then we prove that $v^\ep (x,t)$ directly converges in the $L^1$ sense to a unique kinetic solution $u(x,t)$ of (\ref{dp1}), (\ref{dp2}) as $\ep \to 0$. In order to discuss this problem in the $L^1$ setting (not in the BV setting) we need to choose an appropriate partition $\{ t^\ep_n \}$ of $[0,T]$; we can not take a partition arbitrarily. It is shown in Section 3 that such a desirable partition indeed exists. Then we prove in Section 4 that $v^\ep (x,t)$ converges strongly in $L^1$ to $u (x,t)$ as $\ep \to 0$ by using the technique of ``doubling of variables" (\cite[Proposition 3.2]{DeVo}).


\indent We now give the precise assumptions in this paper:
\begin{enumerate}
\item[${\rm (H_1)}$] The flux function $B: \mathbb{R} \to \mathbb{R}^d$ is of class $C^2$ and its derivatives denoted by $b = (b_1 , \ldots , b_d)$ have at most polynomial growth.
\item[${\rm (H_2)}$] The diffusion matrix $A = (A_{ij}): \mathbb{R} \to \mathbb{R}^{d \times d}$ is symmetric and positive semidefinite. Its square-root matrix denoted by $\si$ is also symmetric and positive semidefinite. We assume that $\si$ is bounded and locally $\gamma$-H\"older continuous for some $\gamma > 1/2$.
\item[${\rm (H_3)}$] For each $z \in L^2 (\mT^d)$, $\Phi (z) : H \to L^2 (\mT^d)$ is defined by $\Phi (z) e_k = g_k (\cdot, z(\cdot))$, where $g_k \in C (\mT^d \times \mR)$ satisfies the following conditions:
\begin{gather}
G^2 (x,\xi) = \sum_{k=1}^\infty | g_k (x,\xi)|^2 \leq C(1+ | \xi |^2), \label{H3} 
\\
\sum_{k=1}^\infty | g_k (x,\xi) - g_k (y,\zeta) |^2 \leq C \Big( |x-y|^2 + |\xi - \zeta | r(|\xi - \zeta |) \Big) \label{H4} 
\end{gather}for every $x,y \in \mT^d$, $\xi, \zeta \in \mathbb{R}$. Here, $C$ is a constant and $r$ is a continuous non-decreasing function on $\mathbb{R}_+$ with $r(0) = 0$.
\end{enumerate}These assumptions are the same as that of \cite{DeHoVo}, but it is assumed in \cite{DeHoVo} that the function $r$ also satisfies the following additional condition:
\begin{gather*}
 r (\de) \leq C \de^\al, \hs{4mm} \de < 1
\end{gather*}for some $\al > 0$.


\indent This paper is organized as follows. In Section 2, we introduce the notion of kinetic solutions to (\ref{dp1}), (\ref{dp2}) by using the kinetic formulation and state the main result of existence. We construct approximate solutions to (\ref{dp1}), (\ref{dp2}) and give some fundamental lemmas concerning these approximations. In Section 3 we show that the approximate solutions are indeed defined on the whole interval $[0,T)$. Section 4 is devoted to the proof of the main theorem.

\section{Preliminaries and the main result} 

We give the definition of solution in this section. Define
\begin{gather*}
f^+ (u,\xi)= \begin{cases}
1 & \text{if $\xi < u$}, \\
0 & \text{if $\xi \geq u$},
\end{cases}
\quad \text{and} \quad f^- (u,\xi)= \begin{cases}
-1 & \text{if $\xi > u$}, \\
0 & \text{if $\xi \leq u$}.
\end{cases}
\end{gather*}


 Hereafter, we will use the notation $A:B = \sum_{i,j} a_{ij} b_{ij}$ for two matrices $A = (a_{ij})$, $B = (b_{ij})$ of the same size.


\begin{defn}[Kinetic measure]
A map $m$ from $\Omega$ to $\mathcal{M}_b^+(\mT^d \times [0,T) \times \mathbb{R})$, the set of non-negative finite measures over $\mT^d \times [0,T) \times \mathbb{R}$, is said to be a kinetic measure if
\begin{enumerate}
\item[$\mathrm{(i)}$] $m$ is weakly measurable, i.e., for each $\phi \in C_b (\mT^d \times [0,T) \times \mR)$ the map $m(\phi): \Om \to \mR$ is measurable,
\item[$\mathrm{(ii)}$] $m$ vanishes for large $\xi$ in the following sense:
\begin{gather}
 \lim_{R \to \infty} \mE m (\mT^d \times [0,T) \times \{ \xi \in \mR; R \leq |\xi| \}) = 0, \label{decay} 
\end{gather}
\item[$\mathrm{(iii)}$] for all $\phi \in C_b (\mT^d \times \mathbb{R})$, the process
\begin{gather}
 t \mapsto \int_{\mT^d \times [0,t] \times \mathbb{R}} \phi (x,\xi) \ dm (x, s ,\xi) \label{predictable} 
\end{gather}is predictable.
\end{enumerate}
\end{defn}


\begin{defn}[Kinetic solution]
Let $u_0 \in L^p(\Omega, \mathscr{F}_0, dP ;L^p (\mT^d))$ and $u \in L^p (\Om \times [0,T), \mathcal{P}, dP \otimes dt; L^p (\mT^d)) \cap L^p (\Omega; L^\infty (0,T; L^p (\mT^d)))$ for all $p \in [1,\infty)$, where $\mathcal{P}$ is the predictable $\sigma$-algebra on $\Omega \times [0,T)$ associated to $(\mathcal{F}_t)$. Assume that ${\rm div} \int_0^u \si (\zeta) \ d\zeta \in L^2 (\Om \times \mT^d \times [0,T))$ and for any $\phi \in C_b (\mR)$ the following chain rule formula holds true:
\begin{gather}
 {\rm div} \int_0^u \phi (\zeta) \si (\zeta) \ d\zeta = \phi (u) \hs{0.5mm} {\rm div} \int_0^u \si (\zeta) \ d\zeta \hs{5mm} {\rm in} \hs{5mm} \mathcal{D}' (\mT^d) \hs{3mm} {\rm a.e.} \ (\omega, t). \label{chain} 
\end{gather}Let $n_1: \Om \to \mathcal{M}_b^+ (\mT^d \times [0,T) \times \mR)$ be defined as follows: for any $\phi \in C_b (\mT^d \times [0,T) \times \mR)$,
\begin{gather}
n_1 (\phi) = \int_0^T \int_{\mT^d} \int_\mR \phi (x,t,\xi) \left| {\rm div} \int_0^u \si (\zeta) \ d\zeta \right|^2 d \de_{u (x,t)} (\xi) dx dt,\label{measure} 
\end{gather}where $\de_u$ is the Dirac measure centered at $u$. \\
Then $u$ is said to be a kinetic solution to (\ref{dp1})-(\ref{dp2}) with initial datum $u_0$ if there exists a kinetic measure $m \geq n_1$, $P$-a.s., such that the pair $(u,m)$ satisfies a kinetic formulation: for all $\varphi \in C_c^\infty (\mT^d \times [0,T) \times \mR)$, $P$-a.s.,
\begin{align}
& \int_0^T \int_{\mT^d} \int_\mathbb{R} f^\pm (u,\xi) (\partial_t + b(\xi) \cdot \nabla + A (\xi): D^2) \varphi \  d \xi dx dt \nonumber \\
& \hs{7mm} + \int_{\mT^d} \int_\mathbb{R} f^\pm (u_0, \xi) \varphi (0) \ d\xi dx \nn \\
& = - \sum_{k=1}^\infty \int_0^T \int_{\mT^d} g_k (x,u) \varphi (x,t,u) \ dx d\beta_k (t) \nn \\
& \hs{7mm} - \frac{1}{2} \int_0^T \int_{\mT^d} G^2 (x,u) \pa_\xi \ph (x,t,u) \ dx dt + \int_{\mT^d \times [0,T) \times \mathbb{R}} \partial_\xi \varphi \ dm. \label{kine1} 
\end{align}
\end{defn}


\begin{rem}
The above definition concerning the notion of kinetic solution has been introduced in \cite{DeHoVo}. For the advantage of kinetic solutions as well as kinetic formulations in the stochastic case, we refer to \cite{DeHoVo}, \cite{DeVo}, \cite{DeVo2}, \cite{Ho}.
\end{rem}


We are now in a position to state our main result.


\begin{thm} \label{main}
 Let $u_0 \in L^p (\Om; L^p (\mT^d))$ for all $p \in [1,\infty)$. Under the assumptions ($H_1$), ($H_2$), ($H_3$), there exists a unique kinetic solution to (\ref{dp1}), (\ref{dp2}), which has almost surely continuous orbits in $L^p (\mT^d)$. Moreover,
 \begin{gather*}
  \mE \| u_1 (t) - u_2 (t) \|_{L^1 (\mT^d)} \leq \mE \| u_{1,0} - u_{2,0} \|_{L^1 (\mT^d)}
 \end{gather*}for all kinetic solutions $u_1$, $u_2$ to (\ref{dp1}), (\ref{dp2}) with initial data $u_{1,0}$ and $u_{2,0}$, respectively.
\end{thm}


\begin{rem}
 The uniqueness has been proved in \cite{DeHoVo} while the existence is obtained under the further additional assumption that for some $\al > 0$, $r(\de) \leq C \de^\al$, $\de < 1$. Namely we here obtain the existence under the same assumptions that are assumed in order to obtain the uniqueness.
\end{rem}


Let us now explain the construction and some properties of the approximate solutions. We consider the following two equations: for $0 \leq s < t \leq T$,
\begin{align}
\begin{cases}
dv = \Phi (v) d W (t) \\
v(\cdot ,s) = v_s (\cdot)
\end{cases} \label{R} 
\end{align}and
\begin{align}
\begin{cases}
\pa_t u + {\rm div} (B (u)) = {\rm div} (A (u) \nabla u) \\
u(\cdot, s) = u_s (\cdot).
\end{cases} \label{S} 
\end{align}Let $R(t,s)$ and $S(t-s)$ be the solution operators of (\ref{R}) and (\ref{S}), respectively. Namely we can write
\begin{gather*}
 v(t,s) = R (t,s)v_s \hs{4mm} \text{and} \hs{4mm} u(t,s) = S (t-s)u_s.
\end{gather*}


For the SDE (\ref{R}) we have


\begin{lemm} \label{lemmaSDE} 
 Let $v_s \in L^p (\Om ; \mathscr{F}_s, dP; L^p (\mT^d))$ for $p \geq 1$. There exists a unique kinetic solution $v(t,s)$ to (\ref{R}), which has a representative in $L^p (\Om; L^\infty (s,T; L^p (\mT^d)))$ with almost surely continuous trajectories in $L^p (\mT^d)$. Besides, it satisfies the following ``strong" kinetic formulation at all $t \in [s,T)$, that is, weak in $(x,t)$ only: $P$-a.s., for all $t \in [s,T)$, for all $\ph \in C_c^\infty (\mT^d \times \mR)$,
 \begin{align}
  & \hs{-5mm} - \int_{\mT^d} \int_\mR f^\pm (v(t,s), \xi) \ph \ d\xi dx + \int_{\mT^d} \int_\mR f^\pm (v_s, \xi) \ph \ d\xi dx  \nn \\
  & = - \sum_{k=1}^\infty \int_s^t \int_{\mT^d} g_k (x,v(t,s)) \ph (x,v(r,s)) \ dx d\beta_k (r) \nn \\
  & \hs{5mm} - \frac{1}{2} \int_s^t \int_{\mT^d} G^2 (x,v(r,s)) \pa_\xi \ph (x,v(r,s)) \ dx dr. \label{stochast}
 \end{align}
\end{lemm}


\begin{proof}
 Let us recall the existence proof of \cite[Section 4]{DeVo2}, and so $v_s^\eta$ is a smooth approximation of $v_s$, $\Phi_\eta$ is a suitable Lipschitz approximation of $\Phi$ satisfying (\ref{H3}), (\ref{H4}) uniformly, $g^\eta_k$ and $G^\eta$ are defined as in the case $\eta = 0$. Moreover, we may assume that $g^\eta_k \in C_c^\infty (\mT^d \times \mR)$ and $g^\eta_k = 0$ for $k \geq 1/\eta$.
 

In the case of no flux one can take the following approximations instead of parabolic approximations:
\begin{align*}
\begin{split}
  & dv^\eta =  \Phi_\eta (v^\eta) \ dW (t), \\
  & v^\eta (x,0) =  v^\eta_s (x),
\end{split} \hs{-15mm}
\begin{split}
  & (x,t) \in \mT^d \times (s,T), \\
  & x \in \mT^d.
\end{split}
\end{align*}It is shown in \cite[Theorem 7.4]{DaZa} (also see \cite[Theorem 3.15]{GaMa}) that the above SDE has a unique $L^2 (\mT^d)$ valued continuous mild solution $v^\eta$. Thanks to \cite[Theorem 3.2]{GaMa}, $v^\eta$ is indeed a strong solution. Moreover, it is also shown in \cite{DaZa} and \cite{GaMa} that using the It\^o formula and the Gronwall lemma one can obtain
\begin{gather*}
 \mE \| v^\eta (t) \|^2_{L^2 (\mT^d)} \leq C(T) \big( 1 + \mE \| v^\eta_s \|^2_{L^2 (\mT^d)}  \big).
\end{gather*}Also, for $p \geq 2$, by the It\^o formula applied to $|v|^p$ and a martingale inequality
\begin{gather*}
 \mE \big( \sup_{ t \in [s,T)} \| v^\eta \|^p_{L^p (\mT^d)} \big) \leq C(p, v^\eta_s, T).
\end{gather*}From now on we can proceed with the proof by the same manner as in \cite[Section 4]{DeVo2}. However, we note that in this case the kinetic measure $m^\eta = \eta | \nabla v^\eta |^2 \de_{v^\eta = \xi}$ disappears owing to no viscosity $\eta \Delta v^\eta$ in the approximate equation; and besides, due to the fact that $v^\eta$ is sufficiently regular one sees that it satisfies the corresponding kinetic formulation without any measure. Therefore there exists a kinetic solution $v(t,s)$ to (\ref{R}) with the kinetic measure $m = 0$. Furthermore, by virtue of the uniqueness and reduction theorem (\cite[Theorem 15, Corollary 16]{DeVo2}) we see that $v(t,s)$ has a representative in $L^p (\Om; L^\infty (s,T; L^p (\mT^d)))$ with almost surely continuous trajectories in $L^p (\mT^d)$. In particular, we can regard $v(\cdot, s)$ as a function of $C([s,T); L^1(\mT^d))$ $P$-a.s.. Let $t \in [s,T)$ and let $\{ t_j \}$ be a sequence in $[s,T)$ such that $\lim_{j \to \infty} t_j = t$ and $P$-a.s., $\lim_{j \to \infty} v(t_j,s) = v(t,s)$ a.e. $x \in \mT^d$. From the first part of the proof of Theorem 6.4 in \cite{DeHoVo}, we find that $P$-a.s., $\lim_{j \to \infty} f^\pm (v(t_j, s), \xi) = f^\pm (v(t,s), \xi)$ weakly-$*$ in $L^\infty (\mT^d \times \mR)$. This inplies that $f(v(\cdot, s),\xi)$ is weakly-$*$ continuous in $L^\infty (\mT^d \times \mR)$, $P$-a.s.. Hence the strong kinetic formulation at every $t \in [s,T)$ follows from the kinetic formulation (7) as in \cite{DeVo2} which has no flux and no kinetic measure by taking appropriate test functions $\ph$ (see \cite[equation (22)]{DeVo2}).

\end{proof}


\begin{lemm} \label{lem} 
 For $p \geq 2$,
 \begin{gather}
  \mE \| v (t,s) \|_{L^p (\mT^d)}^p \leq e^{K (t-s)} \left(\mE \| v_s \|_{L^p (\mT^d)}^p + K_T (t-s) \right),
 \end{gather}where $K$ is a constant depending on $p$ and $C$ appearing in (\ref{H3}) and $K_T$ a constant depending on $T$ as well.
\end{lemm}


\begin{proof}
 It suffices to show the estimate for $s=0$, and write $v(t) = v(t,0)$. Take $\ph_n^\pm (\xi) = p |\xi|^{p-1} (\xi)^\pm \Psi_n (\xi)$ as a test function in (\ref{stochast}), where $\Psi_n$ is a cut off function on $\mR$ defined by $\Psi_n (\xi) = 1$ if $|\xi| \leq n$, $\Psi_n (\xi) = 0$ if $|\xi| \geq 2n$ and $| \Psi'_n (\xi) | \leq \frac{C}{n}$. Letting $n \to \infty$, summing (\ref{stochast}) for $f^+$ and for $f^-$, we have
 \begin{align*}
  & \| v (t) \|^p_{L^p (\mT^d)} = \| v (0) \|^p_{L^p (\mT^d)} + p \int_0^t \int_{\mT^d} | v (s) |^{p-2} v(s) \Phi (v(s)) \ dx d\beta_k (s) \\
  & \hs{23mm} + \frac{1}{2} p (p-1) \int_0^t \int_{\mT^d} | v (s) |^{p-2} \sum_{k=1}^\infty | g_k (x,v) |^2 \ dxds,
 \end{align*}for almost surely, for all $t \in [0,T)$. Taking expectation and using (\ref{H3}) we have
 \begin{gather*}
  \mE \| v(t) \|_{L^p (\mT^d)}^p \leq \mE \| v (0) \|_{L^p (\mT^d)}^p + Kt + K \int_0^t \mE \| v (s) \|_{L^p (\mT^d)}^p \ ds,
 \end{gather*}which, together with the Gronwall inequality, yields the desired estimate.
\end{proof}


On the other hand, Chen and Perthame \cite{ChPe} has already proved the well-posedness of the deterministic anisotropic degenerate parabolic equation (\ref{S}): For each $u_s (\cdot) \in L^1 (\mT^d)$ there exists a unique kinetic solution $u (t,s) \in C ([s,T); L^1 (\mT^d))$ in the sense of the deterministic version of Definition 2.2 (see \cite[Definition 2.2]{ChPe} and Remark \ref{rem2.3} below). Besides we have for all $t \in [s,T)$ and $p \in [1,\infty]$,
\begin{gather}
 \| u_1 (t) - u_2 (t) \|_{L^1 (\mT^d)} \leq \|  u_1 (s) - u_2 (s) \|_{L^1 (\mT^d)}, \label{contraction} \\
 \| u_1 (t) \|_{L^p (\mT^d)} \leq \| u_1 (s) \|_{L^p (\mT^d)}, \label{Lp} 
\end{gather}where $u_i$, $i = 1,2$, are arbitrary kinetic solutions to (\ref{S}).


 To prove our existence theorem we propose to approximate the equations (\ref{dp1})-(\ref{dp2}) as follows. Let $\ep > 0$ and let $t_0^\ep = 0$, $\tilde{u}_{0}^\ep = u_0$. For $n \in \mN \cup \{ 0 \}$, if $t_n^\ep < T$, define
\begin{align*}
& t^\ep_{n+1} := \inf \{ t> t^\ep_n ; \ \mE \| S(t - t^\ep_n) \tilde{u}^\ep_{n} - \tilde{u}^\ep_{n} \|_{L^1 (\mT^d)} > \ep \} \wedge ( {t^\ep_n + \ep} ) \wedge T, \\
& u_n^\ep := S(t_{n+1}^\ep - t_n^\ep ) \tilde{u}_{n}^\ep, \hs{4mm} \tilde{u}_{n+1}^\ep := R(t_{n+1}^\ep, t_n^\ep) u_n^\ep;
\end{align*}if $t_n^\ep = T$, define $t_{n+1}^\ep = T$ where $a \wedge b = \min \{a,b \}$. Then define the approximate solutions $v^\ep$ and $\tilde{v}^\ep$ by
\begin{align*}
& v^\ep (t) := R(t,t_n^\ep) u_n^\ep \hs{4mm} \text{for} \ t \in [t_n^\ep, t_{n+1}^\ep ) \hs{4mm} \text{a.s.,} \\
& \tilde{v}^\ep (t) := S(t-t^\ep_n) \tilde{u}^\ep_{n} \hs{4mm} \text{for} \ t \in [t_n^\ep, t_{n+1}^\ep ) \hs{4mm} \text{a.s.}
\end{align*} Set $T^\ep = \sup_{n \geq 1} t^\ep_n$. Obviously, $v^\ep$ and $\tilde{v}^\ep$ are the functions defined on $[0,T^\ep)$ such that
\begin{align*}
 & v^\ep (x,t_n^\ep) = u^\ep_n (x), \hs{4mm} v^\ep (x,t_{n+1}^\ep -0) = \tilde{u}^\ep_{n+1} (x), \\
 & \tilde{v}^\ep (x,t_n^\ep) = \tilde{u}^\ep_n (x), \hs{4mm} \tilde{v}^\ep (x, t_{n+1}^\ep - 0) = u_{n}^\ep (x).
\end{align*}By virtue of Lemma \ref{lem} and (\ref{Lp}) $v^\ep$ and $\tilde{v}^\ep$ satisfy the following estimates, respectively: for any $p \geq 1$ there exists a constant $C$ depending on $p$, on the terminal $T$ and on the initial condition $u_0$ but not on $\ep$ such that for all $t \in [0,T^\ep )$, 
\begin{gather}
 \mE \| v^\ep (t) \|_{L^p (\mT^d)}^p \leq C, \label{estimate01} 
  \\
 \mE \| \tilde{v}^\ep (t) \|_{L^p (\mT^d)}^p \leq C. \label{estimate02} 
\end{gather}Indeed, in the case of $v^\ep$, we can calculate as follows. For any $t \in [0,T^\ep )$ there exists a interval $[ t^\ep_n, t^\ep_{n+1} )$ which includes $t$. Then using Lemma \ref{lem} and (\ref{Lp}) repeatedly, we have
\begin{align*}
 & \hs{-5mm} \mE \| v^\ep (t) \|^p_{L^p (\mT^d)} \\
 & \leq e^{K (t^\ep_{n+1} - t^\ep_n)} \mE \| u^\ep_n \|_{L^p (\mT^d)}^p + e^{K (t^\ep_{n+1} - t^\ep_n)} K_T (t^\ep_{n+1} - t^\ep_n) \\
 & \leq e^{K (t^\ep_{n+1} - t^\ep_n)} \mE \| \tilde{u}^\ep_n \|_{L^p (\mT^d)}^p + e^{K (t^\ep_{n+1} - t^\ep_n)} K_T (t^\ep_{n+1} - t^\ep_n) \\
 & \leq e^{K (t^\ep_{n+1} - t^\ep_{n-1})} \mE \| u^\ep_{n-1} \|_{L^p (\mT^d)}^p \\
 & \hs{5mm} + e^{K (t^\ep_{n+1} - t^\ep_{n-1})} K_T (t^\ep_{n+1} - t^\ep_{n-1}) + e^{K (t^\ep_{n+1} - t^\ep_{n})} K_T (t^\ep_{n+1} - t^\ep_{n}) \\
 & \leq \cdots \\
 & \leq e^{K t^\ep_{n+1}} \mE \| u_0 \|_{L^p (\mT^d)}^p + \sum_{k=0}^n e^{K (t^\ep_{n+1} - t^\ep_k)} K_T (t^\ep_{k+1} - t^\ep_{k}) \\
 & \leq e^{KT} \left( \mE \| u_0 \|_{L^p (\mT^d)}^p + K_T T \right).
\end{align*}Thus we obtain the estimate (\ref{estimate01}). The estimate (\ref{estimate02}) will be proved in the same fashion.


We now derive the kinetic formulation satisfied by the approximate solutions $v^\ep$, $\tilde{v}^\ep$. Let $\ph \in C_c^\infty (\mT^d \times \mR)$. $v^\ep$ satisfies the strong kinetic formulation at every $t \in [t^\ep_n, t^\ep_{n+1})$ by Lemma \ref{lemmaSDE}: $P$-a.s., for all $t \in [t^\ep_n, t^\ep_{n+1})$, 
\begin{align}
& \hs{-7mm} -\int_{\mT^d} \int_\mR f^\pm (v^\ep(t), \xi) \ph \ d\xi dx + \int_{\mT^d} \int_{\mR} f^\pm (u_n^\ep, \xi) \ph \ d\xi dx \nn \\
& = - \sum_{k=1}^\infty \int_{t_n^\ep}^t \int_{\mT^d} g_k (x,v^\ep) \ph(x,v^\ep) \ dx d\beta_k (s) \nn \\
& \hs{7mm} - \frac{1}{2} \int_{t^\ep_n}^t \int_{\mT^d} G^2 (x,v^\ep) \pa_\xi \ph (x,v^\ep) \ dx ds. \label{a1} 
\end{align}On the other hand, note that $\tilde{v}^\ep \in C ([t^\ep_n, t^\ep_{n+1}); L^1 (\mT^d))$ as stated after the proof of Lemma \ref{lem}. Hence $\tilde{v}^\ep$ satisfies the strong kinetic formulation at every $t \in [t^\ep_n, t^\ep_{n+1})$ (see Remark \ref{rem2.3} below):
\begin{align}
& -\int_{\mT^d} \int_\mR f^\pm (\tilde{v}^\ep(t), \xi) \ph \ d\xi dx + \int_{\mT^d} \int_{\mR} f^\pm (\tilde{u}_n^\ep, \xi) \ph \ d\xi dx \nn \\
& \hs{7mm} + \int_{t_n^\ep}^t \int_{\mT^d} \int_\mR f^\pm (\tilde{v}^\ep (s), \xi) (b(\xi) \cdot \nabla + A(\xi): D^2) \ph \ d\xi dx ds \nn \\
& = \int_{\mT^d \times [t^\ep_n, t] \times \mR} \pa_\xi \ph \ dm_n^\ep, \label{a2} 
\end{align}$P$-a.s., for all $t \in [t_n^\ep , t_{n+1}^\ep)$, where $m^\ep_n$ is the associated entropy dissipation measure on $\mT^d \times [t_n^\ep, t_{n+1}^\ep) \times \mR$, a.s. such that $m^\ep_n \geq n^\ep_{1,n}$ and
\begin{gather*}
 \lim_{R \to \infty} m^\ep_n (\mT^d \times [t_n^\ep, t_{n+1}^\ep) \times \{ \xi \in \mR; R \leq |  \xi | \}) = 0, \hs{3mm} \text{a.s.},
\end{gather*}and $n^\ep_{1,n}$ is the parabolic dissipation measure on $\mT^d \times [t^\ep_n, t^\ep_{n+1}) \times \mR$ which is defined by (\ref{measure}) with $u$ replaced by $\tilde{v}^\ep$. Letting $t \uparrow t^\ep_{n+1}$ in the kinetic formulations (\ref{a1}) and (\ref{a2}), we have
\begin{align}
& \hs{-7mm} -\int_{\mT^d} \int_\mR f^\pm (\tilde{u}^\ep_{n+1}, \xi) \ph \ d\xi dx + \int_{\mT^d} \int_{\mR} f^\pm (u_n^\ep, \xi) \ph \ d\xi dx \nn \\
& = - \sum_{k=1}^\infty \int_{t_n^\ep}^{t^\ep_{n+1}} \int_{\mT^d} g_k (x,v^\ep) \ph(x,v^\ep) \ dx d\beta_k (s) \nn \\
& \hs{7mm} - \frac{1}{2} \int_{t^\ep_n}^{t^\ep_{n+1}} \int_{\mT^d} G^2 (x,v^\ep) \pa_\xi \ph (x,v^\ep) \ dx ds, \label{b1} 
\end{align}and
\begin{align}
& -\int_{\mT^d} \int_\mR f^\pm (\tilde{u}^\ep_n, \xi) \ph \ d\xi dx + \int_{\mT^d} \int_{\mR} f^\pm (\tilde{u}_n^\ep, \xi) \ph \ d\xi dx \nn \\
& \hs{7mm} + \int_{t_n^\ep}^{t^\ep_{n+1}} \int_{\mT^d} \int_\mR f^\pm (\tilde{v}^\ep (s), \xi) (b(\xi) \cdot \nabla + A(\xi): D^2) \ph \ d\xi dx ds \nn \\
& = \int_{\mT^d \times [t^\ep_n, t^\ep_{n+1}) \times \mR} \pa_\xi \ph \ dm_n^\ep. \label{b2} 
\end{align}Now take any $t \in [0,T^\ep )$. Then there exists an interval $[t^\ep_n, t^\ep_{n+1})$ which includes $t$. Summing (\ref{b1}), (\ref{b2}) over $0$ to $n-1$ and then summing (\ref{a1}), (\ref{a2}) for $t \in [t^\ep_n, t^\ep_{n+1})$, we have
\begin{align}
& - \int_{\mT^d} \int_\mR f^\pm (v^\ep (t), \xi) \ph \ d\xi dx - \int_{\mT^d} \int_\mR f^\pm ( \tilde{v}^\ep (t), \xi) \ph \ d\xi dx \nn \\
& \hs{7mm} + \int_{\mT^d} \int_\mR f^\pm ( {v}^\ep (t^\ep), \xi) \ph \ d\xi dx + \int_{\mT^d} \int_\mR f^\pm (u_0, \xi) \ph \ d\xi dx \nn \\
& \hs{7mm} + \int_0^t \int_{\mT^d} \int_{\mR} f^\pm (\tilde{v}^\ep (s), \xi) (b (\xi) \cdot \nabla + A (\xi):D^2 ) \ph \ d\xi dx ds \nn \\
& = - \sum_{k=1}^\infty \int_0^t \int_{\mT^d} g_k (x,v^\ep) \ph (x,v^\ep) \ dx d\beta_k (s) \nn \\
& \hs{7mm} -\frac{1}{2} \int_0^t \int_{\mT^d} G^2 (x,v^\ep) \pa_\xi \ph (x,v^\ep) \ dx ds + \int_{[0,t] \times \mT^d \times \mR} \pa_\xi \ph \ dm^\ep, \label{kine} 
\end{align}a.s., for $\ph \in C^\infty_c (\mT^d \times \mR)$ and $t \in [0,T^\ep)$, where we have used the notations that $m^\ep = \sum_{n=0}^{\infty} m^\ep_n$ and $t^\ep = t^\ep_{k}$ if $t \in [t^\ep_k, t^\ep_{k+1})$, $k \in \mN \cup \{ 0 \}$. We note that $m^\ep$ is almost surely finite by virtue of the following lemma.


\begin{lemm} \label{lemma2}
 For any $p \geq 0$, there exists a constant $C = C(p,u_0,T) \geq 0$ such that for all $\ep \in (0,1)$ the measure $m^\ep$ satisfies the following estimate:
 \begin{gather}
  \mE \int_{\mT^d \times [0,T^\ep) \times \mR} |\xi|^p \ d m^\ep (x,t,\xi) \leq C. \label{Lpbound} 
 \end{gather}
\end{lemm}


\begin{proof}
 From (\ref{a1}) and (\ref{a2}), we deduce the following type of a kinetic formulation:
\begin{align}
 & - \int_{\mT^d} \int_\mR \chi (v^\ep (t), \xi) \ph \ d\xi dx + \int_{\mT^d} \int_\mR \chi (u_0, \xi) \ph \ d\xi dx \nn \\
 & \hs{7mm} + \int_0^{t^\ep} \int_{\mT^d} \int_\mR \chi (\tilde{v}^\ep, \xi) (b(\xi) \cdot \nabla + A (\xi) : D^2) \ph \ d\xi dx ds \nn \\
 & = - \sum_{k=1}^\infty \int_0^t \int_{\mT^d} g_k (x, v^\ep) \ph (x,v^\ep) \ dx d\beta_k (s) \nn \\
 & \hs{7mm} - \frac{1}{2} \int_0^t \int_{\mT^d} G^2 (x,v^\ep) \pa_\xi \ph (x,v^\ep) \ dx ds + \int_{\mT^d \times [0,t^\ep] \times \mR} \pa_\xi \ph \ dm^\ep, \label{kinetic} 
\end{align}a.s., for $\ph \in C^\infty_c (\mT^d \times \mR)$ and $t \in [0,T^\ep)$, where $\chi$ is an equilibrium function, i.e., $\chi (w,\xi) = 1$ if $0 < \xi < w$, $\chi (w,\xi) = -1$ if $w < \xi < 0$ and $\chi (w,\xi) = 0$ otherwise. Since for any $p \geq 0$, $v^\ep (t)$ and $\tilde{v}^\ep (t)$ are functions of $L^{p+2} (\mT^d)$ a.s. by virtue of (\ref{estimate01}) and (\ref{estimate02}), we can take $(x,\xi) \mapsto \frac{1}{p+1} |\xi|^{p} \xi$ as a test function. We then get
 \begin{align}
  & \int_{\mT^d \times [0,t^\ep] \times \mR} |\xi|^p \ dm^\ep (x,t,\xi) \nn \\
  & \hs{4mm} = - \frac{1}{(p+1)(p+2)} \| v^\ep (t) \|_{L^{p+2} (\mT^d)}^{p+2} + \frac{1}{(p+1)(p+2)} \| u_0 \|_{L^{p+2} (\mT^d)}^{p+2} \nn \\
  & \hs{8mm} + \frac{1}{p+1} \sum_{k=1}^\infty \int_0^t \int_{\mT^d} g_k (x,v^\ep (s)) \left| v^\ep (s) \right|^p v^\ep (s) \ dx d\beta_k (s) \nn \\
  & \hs{8mm} + \frac{1}{2} \int_0^t \int_{\mT^d} G^2 (x,v^\ep (s)) \left| v^\ep (s) \right|^p \ dx ds. \label{calc2} 
 \end{align}Taking expectation, from the assumption (\ref{H3}) and the estimate (\ref{estimate01}), we deduce
 \begin{gather*}
  \mE \int_{\mT^d \times [0,t^\ep] \times \mR} |\xi|^p \ dm^\ep (x,s,\xi) \leq C,
 \end{gather*}for any $t \in [0,T^\ep)$. Then letting $t \uparrow T^\ep$, we get the conclusion by Fatou's lemma.
\end{proof}


\begin{lemm} \label{lemma1}
 For all $p \in [2, \infty)$ there exists a constant $C = C(p, u_0, T) \geq 0 $ such that for all $\ep \in (0,1)$ the approximate solutions $v^\ep$ and $\tilde{v}^\ep$ satisfy the following energy inequalities:
 \begin{gather}
  \mE \sup_{t \in [0,T^\ep)} \| v^\ep (t) \|^p_{L^p (\mT^d)} \leq C, \label{apriori1} 
\\
  \mE \sup_{t \in [0,T^\ep)} \| \tilde{v}^\ep (t) \|^p_{L^p (\mT^d)} \leq C. \label{apriori2} 
 \end{gather}
\end{lemm}


\begin{proof}
 We prove the lemma only in the case of $v^\ep$. The case of $\tilde{v}^\ep$ will be done in a similar fashion. We can take $(x,\xi) \mapsto p |\xi|^{p-2} \xi$ for $p \geq 2$ as a test function in (\ref{kinetic}). Then we get
\begin{align*}
 & \hs{-4mm} - \| v^\ep (t) \|_{L^p (\mT^d)}^p + \| u_0 \|_{L^p (\mT^d)}^p \nn \\
 & \hs{4mm} = - p \sum_{k=1}^\infty \int_0^t \int_{\mT^d} g_k (x,v^\ep (s)) | v^\ep (s) |^{p-2} v^\ep (s) \ dx d\beta_k (s) \nn \\
 & \hs{8mm} - \frac{p(p-1)}{2} \int_0^t \int_{\mT^d} G^2 (x,v^\ep (s)) | v^\ep (s) |^{p-2} \ dx ds \nn \\
 & \hs{8mm} + p (p-1) \int_{\mT^d \times [0,t^\ep] \times \mR} | \xi |^{p-2} \ dm^\ep (x,s,\xi). 
\end{align*}
After dropping the term of the non-negative measure $m^\ep$, take supremum, expectation and use the assumption (\ref{H3}) to obtain
\begin{align*}
 & \mE \sup_{ t \in [0,T^\ep) } \| v^\ep (t) \|^p_{L^p (\mT^d)} \leq \mE \| u_0 \|^p_{L^p (\mT^d)} + C \left( 1 + \int_0^{T^\ep} \| v^\ep (s) \|_{L^p (\mT^d)}^p \ ds \right) \\
 & \hs{12mm} + p \mE \sup_{t \in [0,T^\ep )} \left| \sum_{k=1}^\infty \int_0^t \int_{\mT^d} g_k (x,v^\ep (s)) \left| v^\ep (s) \right|^{p-2} v^\ep (s) \ dx d\beta_k (s) \right|.
\end{align*}For the term of stochastic integral we employ the Burkholder-Davis-Gundy inequality, the Schwarz inequality, the assumption (\ref{H3}) and also the weighted Young inequality to obtain
\begin{align}
 & \hs{-2mm} p \mE \sup_{t \in [0,T^\ep )} \left| \sum_{k=1}^\infty \int_0^t \int_{\mT^d} g_k (x,v^\ep (s)) \left| v^\ep (s) \right|^{p-2} v^\ep (s) \ dx d\beta_k (s) \right| \nn \\
 & \hs{4mm} \leq C \mE \left( \sum_{k=1}^\infty \int_0^{T^\ep} \left( \int_{\mT^d} g_k (x,v^\ep (s)) \left| v^\ep (s) \right|^{p-1} \ dx \right)^2 ds \right)^{\frac{1}{2}} \nn \\
 & \hs{4mm} \leq C \mE \left( \int_0^{T^\ep} \| v^\ep (s) \|_{L^p (\mT^d)}^p \left( \int_{\mT^d} G^2 (x, v^\ep (s)) \left| v^\ep (s) \right|^{p-2} \ dx \right) \ ds \right)^{\frac{1}{2}} \nn \\
 & \hs{4mm} \leq C \mE \left( \sup_{t \in [0,T^\ep )} \| v^\ep (t) \|_{L^p (\mT^d)}^p \right)^{\frac{1}{2}} \left( 1+ \int_0^{T^\ep} \| v^\ep (s) \|_{L^p (\mT^d)}^p \ ds \right)^{\frac{1}{2}} \nn \\
 & \hs{4mm} \leq \frac{1}{2} \mE \sup_{t \in [0,T^\ep )} \| v^\ep (t) \|_{L^p (\mT^d)}^p + C \mE \left( 1 + \int_0^{T^\ep} \| v^\ep (s) \|_{L^p (\mT^d)}^p \ ds \right). \label{calc} 
\end{align}Therefore the conclusion easily follows from (\ref{estimate01}). 
\end{proof}


We now give some properties of the approximate solutions $v^\ep$, $\tilde{v}^\ep$ and the measure $m^\ep$ which follow from Lemmas \ref{lemma2} and \ref{lemma1}. Taking the square of (\ref{calc2}) for $p=0$, then expectation, we get
\begin{gather}
 \mE \left| m^\ep (\mT^d \times [0,T^\ep) \times \mR) \right|^2 \leq C, \label{measure01} 
\end{gather}where a constant $C$ depends on $p$, $T$, $u_0$ but not on $\ep$. Here we estimated the stochastic term in the same manner as that of (\ref{calc}) in order to obtain (\ref{measure01}). On the other hand, since $v^\ep$ satisfies the strong kinetic formulation (\ref{a1}) at every $t \in [t^\ep_n, t^\ep_{n+1})$, we have
\begin{align*}
 \mE \| v^\ep (t) - v^\ep (s) \|_{L^1 (\mT^d)} & = \int_{\mT^d} \mE \left| \sum_{k=1}^\infty \int_s^t g_k (x,v(x,r)) \  d\beta_k (r) \right| dx,
\end{align*}for any $t^\ep_n \leq s \leq t < t^\ep_{n+1}$. We estimate the right hand side of the above equality in the same manner as that of (\ref{calc}) again. Consequently, noting that $t^\ep_{n+1} - t^\ep_n \leq \ep$, one has that for any $n \in \mN \cup \{ 0 \}$, $t^\ep_n \leq s \leq t < t^\ep_{n+1}$
\begin{align}
 \mE \| v^\ep (t) - v^\ep (s) \|_{L^1 (\mT^d)} \leq C T \ep^{1/2} \label{timeL101} 
\end{align}where the constant $C$ depends on only $u_0$, $T$. Moreover, from the construction of the partition $\{ t^\ep_n \}$, it is clear that
\begin{align}
 \mE \| \tilde{v}^\ep (t) - \tilde{v}^\ep (s) \|_{L^1 (\mT^d)} \leq 2 \ep . \label{timeL102} 
\end{align}


\begin{rem} \label{rem2.3} 
 Let $t^* \in (t_n^\ep, t^\ep_{n+1})$. Since $\tilde{v}^\ep \in C([t^\ep_n, t^\ep_{n+1}); L^1 (\mT^d))$, we have that there exists an increasing sequence $\{ t_j \}$ such that $\lim_{j \to \infty} t_j = t^*$ and $\lim_{j \to \infty} \tilde{v}^\ep (t_j) = \tilde{v}^\ep (t^*)$ for a.e. $x \in \mT^d$. From the first part of the proof of Theorem 6.4 in \cite{DeHoVo}, we find that $\lim_{j \to \infty} f^\pm (\tilde{v}^\ep (t_j), \xi) = f^\pm (\tilde{v}^\ep (t^*), \xi)$ weakly-$*$ in $L^\infty (\mT^d \times \mR)$. This means that the left weak-$*$ limit of $f^\pm (\tilde{v}^\ep (t), \xi)$ is equilibrium. Clearly, $f^\pm (\tilde{v})$ is at equilibrium. Therefore, according to \cite[Remark 12]{DeVo2} the associated kinetic measure $m^\ep_n$ in (\ref{a2}) has no atom on $[t^\ep_n, t^\ep_{n+1})$.
\end{rem}


\begin{rem}
 Since the kinetic solution $u$ is defined as a function of $L^p (\Om ; L^\infty (0,T; L^p (\mT^d)))$, $p \geq 1$, the associated kinetic measure might have atoms if the kinetic function $f^\pm (u,\xi)$ was discontinuous with respect to $t$. However, it is shown in \cite[Corollary 3.4]{DeHoVo} that there exists a representatives of $u$ which has, in fact, surely continuous trajectories in $L^p (\mT^d)$. Therefore, one sees that the kinetic measure $m$ has no atom by the same reason as in Remark \ref{rem2.3} above.


In the proof of existence in \cite{DeHoVo} a kinetic solution $u$ is constructed first in $L^p (\Om \times (0,T); L^p (\mT^d)) \cap L^p (\Om; L^\infty (0,T; L^p (\mT^d)) )$ and then it is shown that $u$ belongs to $C ([0,T); L^p (\mT^d))$ almost surely. On the other hand, in the present paper a kinetic solution will be directly constructed in $C ([0,T); L^1 (\mT^d))$.
\end{rem}

\section{The time global approximate solutions} 

In this section, we will show that the partition $\{ t^\ep_n \}$ constructed as above is actually a finite partition of $[0,T)$. 


\begin{prop} \label{longtime}
Let $\ep > 0$. There exists a natural number $M= M(\ep)$ such that $t_M^\ep = T$.
\end{prop}


To this end, we introduce mollifiers on $\mT^d$ and $\mR$ denoted by $(\rho_\eta)$ and by $(\psi_\de)$, respectively, and set $\al_{\eta, \de} = \al_{\eta,\de} (x,y,\xi,\zeta) = \rho_\eta (x-y) \psi_\de (\xi - \zeta)$. We now need the following next two lemmas.


\begin{lemm}[Doubling of variables for (\ref{R})] \label{DVR}
 Let $\ep, \eta, \de > 0$. For each $\tilde{u}^\ep_n$, it holds that
 \begin{align*}
  & - \mE \int_{(\mT^d)^2 \times \mR^2} f^+ (\tilde{u}^\ep_n (x), \xi) f^- (\tilde{u}^\ep_n (y), \zeta) \al_{\eta , \de} \ d\zeta d \xi dy dx \\
  & \hs{4mm} \leq - \mE \int_{(\mT^d)^2 \times \mR^2} f^+ (u^\ep_{n-1} (x), \xi) f^- (u^\ep_{n-1} (y), \zeta) \al_{\eta , \de} \ d\zeta d \xi dy dx \\
  & \hs{7mm} + C (t^\ep_{n} - t^\ep_{n-1}) ( \eta^2 \de^{-1} + r(\de) ),
 \end{align*}where C is a constant independent of $\ep, \eta, \de$. 
\end{lemm}


\begin{proof}
 This Lemma is a special case of \cite[Theorem 15]{DeVo2} such that no flux term appears. In particular see the estimate (35) in \cite{DeVo2} there.
\end{proof}


\begin{lemm}[Doubling of variables for (\ref{S})] \label{DVS}
 Let $\ep, \eta, \de > 0$. For each $u^\ep_n$, it holds that
 \begin{align*}
  & - \mE \int_{(\mT^d)^2 \times \mR^2} f^+ (u^\ep_n (x), \xi) f^- (u^\ep_n (y), \zeta) \al_{\eta , \de} \ d\zeta d \xi dy dx \\
  & \hs{4mm} \leq - \mE \int_{(\mT^d)^2 \times \mR^2} f^+ (\tilde{u}^\ep_{n} (x), \xi) f^- (\tilde{u}^\ep_{n} (y), \zeta) \al_{\eta , \de} \ d \zeta d \xi dy dx \\
  & \hs{7mm} + C (t^\ep_{n} - t^\ep_{n-1}) (\eta^{-1} \de + \eta^{-2} \de^{2 \gamma}) ,
 \end{align*}where C is a constant independent of $\ep, \eta, \de$. 
\end{lemm}


\begin{proof}
This is a case of a restriction of \cite[Theorem 3.3]{DeHoVo} to the deterministic equation. In a similar manner as in the proof of it we can obtain our estimate. Note that we need only the estimate of the terms $I$ and $J$ which appear in there.
\end{proof}


\begin{proof}[Proof of Proposition \ref{longtime}]
 Let $\ep > 0$. By contradiction let us assume that $t^\ep_n < T$ for all $n \in \mN$. Since $\{ t^\ep_n \}$ is increasing sequence, there exists some $\bar{T} \leq T $ such that $t_n^\ep \uparrow \bar{T}$ as $n \to \infty$. Then $\{ u^\ep_n: n \in \mN \}$ is a Cauchy sequence in $L^1 (\Omega \times \mT^d)$. To see this, set for $\eta, \de > 0$ and $n,m \in \mN$, $n < m$,
 \begin{align*}
  & H_\pm (\eta , \de, n, m) = \mE \int_{\mT^d} \left( u^\ep_n (x) - u^\ep_m (x) \right)^\pm \ dx \\
  & \hs{21mm} + \mE \int_{(\mT^d)^2 \times \mR^2} f^\pm (u^\ep_n (x), \xi) f^\mp (u^\ep_m (y), \zeta) \al_{\eta, \de} \ d\zeta d\xi dy dx.
 \end{align*}
  Then by formulations (\ref{a1}) and (\ref{a2})
 \begin{align*}
  & \hs{-2mm} \mE \int_{\mT^d} \left( u^\ep_n (x) - u^\ep_m (x) \right)^\pm \ dx \\
  &  = - \mE \int_{(\mT^d)^2 \times \mR^2} f^\pm (u^\ep_n (x), \xi) f^\mp (u^\ep_n (y), \zeta) \al_{\eta, \de} \ d\zeta d\xi dy dx \\
  & \hs{4mm} - \mE \int_{t^\ep_{n+1}}^{t^\ep_{m+1}} \int_{(\mT^d)^2 \times \mR^2} f^\pm (u^\ep_n (x), \xi) f^\mp (\tilde{v}^\ep (y,s), \zeta) \\
  & \hs{42mm} \times (b(\xi) \cdot \nabla_y + A(\xi):D_y^2) \al_{\eta , \de} \ d\zeta d\xi dy dx ds \\
  & \hs{4mm} - \frac{1}{2} \mE \int_{t^\ep_n}^{t^\ep_m} \int_{(\mT^d)^2 \times \mR^2} f^\pm (u^\ep_n (x), \xi) G^2(y,\zeta) \pa_\zeta \al_{\eta , \de} \ d \de_{v^\ep (y,s)} (\zeta) d\xi dy dx \\
  & \hs{4mm} + \mE \int_{\mT^d \times [t^\ep_{n+1}, t^\ep_{m+1}) \times \mR} \int_{\mT^d \times \mR} f^\pm (u^\ep_n (x), \xi) \pa_\zeta \al_{\eta , \de} \ d\xi dx d m^\ep (y,s,\zeta) \\
  & \hs{4mm} + H_\pm (\eta , \de, n, m).
 \end{align*}We use Lemmas \ref{DVR} and \ref{DVS} repeatedly to get
 \begin{align*}
  & - \mE \int_{(\mT^d)^2 \times \mR^2} f^\pm (u_n^\ep (x), \xi) f^\mp (u_n^\ep (y), \zeta) \al_{\eta, \de} \ d\zeta d\xi dy dx \\
  & \hs{4mm} \leq - \mE \int_{(\mT^d)^2 \times \mR^2} f^\pm (u_0 (x), \xi) f^\mp (u_0 (y), \zeta) \al_{\eta, \de} \ d\zeta d\xi dy dx \\
  & \hs{7mm} + C t^\ep_n (\eta^{-1} \de + \eta^{-2} \de^{2\gamma} + \eta^2 \de^{-1} + r (\de)),
 \end{align*}and hence we get
 \begin{align*}
  & | H_\pm (\eta, \de, n, m)| \\
  & \hs{3mm} \leq \bigg| \mE \int_{(\mT^d)^2 \times \mR} f^\pm (u_n^\ep (x), \xi) \rho_{\eta} (x-y) \Big\{  f^\mp (u^\ep_m (x), \xi) - f^\mp (u^\ep_m (y), \xi) \Big\} \ d\xi dx \bigg| \\
  & \hs{5mm} + \bigg| \int_{(\mT^d)^2 \times \mR^2 } f^\pm (u^\ep_n (x), \xi) \al_{\eta, \de} \Big\{ f^\mp (u^\ep_m (y), \xi) - f^\mp (u^\ep_m (y), \zeta) \Big\} \ d\zeta d\xi dy dx \bigg| \\
  & \hs{3mm} \leq \mE \int_{(\mT^d)^2} \rho_\eta (x-y) \Big| u^\ep_m (x) - u^\ep_m (y) \Big| \ dy dx + \de \\
  & \hs{3mm} \leq -2 \mE \int_{(\mT^d)^2 \times \mR^2} f^+ (u^\ep_m (x),\xi) f^- (u^\ep_m (y), \zeta) \al_{\eta, \de} \ d\zeta d\xi dy dx + 2\de \\
  & \hs{3mm} \leq - 2 \mE \int_{(\mT^d)^2 \times \mR^2} f^+ (u_0 (x), \xi) f^- (u_0 (y), \zeta) \al_{\eta, \de} \ d\zeta d\xi dy dx \\
  & \hs{5mm} + C t^\ep_m (\eta^{-1} \de + \eta^{-2} \de^{2\gamma} + \eta^2 \de^{-1} + r (\de)) + 2\de.
 \end{align*}Therefore one has
 \begin{align*}
  & \mE \int_{\mT^d} \left| u^\ep_n (x) - u^\ep_m (x) \right| \ dx \\
  & \hs{4mm} \leq - 3 \mE \int_{(\mT^d)^2 \times \mR^2} f^+ ( u_0 (x), \xi) f^-( u_0 (y), \zeta)  \al_{\eta, \de} \ d\zeta d\xi dy dx \\
  & \hs{7mm} + C_{\eta , \de} \left( t^\ep_{m+1} - t^\ep_{n+1} + \mE m^\ep (\mT^d \times [t^\ep_{n+1} , t^\ep_{m+1}) \times \mR) \right) \\
  & \hs{7mm} + C t^\ep_m \left( \eta^{-1} \de + \eta^{-2} \de^{2 \gamma} + \eta^2 \de^{-1} + r (\de) \right) + 2\de,
 \end{align*}where $C_{\eta, \de}$ is a constant which depends on $\eta$ and $\de$. Since $\mE m^\ep (\mT^d \times [0,T) \times \mR) \leq C (u_0 , T)$ by Lemma \ref{lemma2}, first letting $n,m \to \infty$ and then letting $\eta \to 0$ with $\de = \eta^\theta$, $\theta \in (1/\gamma, 2)$, yields
 \begin{gather*}
  \lim_{n, m \to \infty} \mE \int_{\mT^d} \left| u^\ep_n (x) - u^\ep_m (x) \right| \ dx = 0.
 \end{gather*}Set $\bar{u}^\ep := \lim_{n \to \infty}u^\ep_n$ and
 \begin{gather*}
  h := \inf \{ s>0; \mE \| S(\bar{T} + s) \bar{u}^\ep - \bar{u}^\ep \|_{L^1 (\mT^d)} > \ep/3 \} \wedge \bar{T}/2.
 \end{gather*}Since $h > 0$, there exists $N_0 \in \mN$ such that for all $n \geq N_0$,
 \begin{gather*}
  \bar{T} - h < t^\ep_n < \bar{T} \hs{4mm} \text{and} \hs{4mm} \mE \| \bar{u}^\ep - u^\ep_n \|_{L^1 (\mT^d)} < \ep/3.
 \end{gather*}It follows from the $L^1$-contraction property (\ref{contraction}) that
 \begin{align*}
  & \hs{-4mm} \mE \| S (t^\ep_{n+1} - t^\ep_n) u^\ep_n - u^\ep_n \|_{L^1 (\mT^d)} \\
  & \leq \mE \| S (t^\ep_{n+1} - t^\ep_n) (u^\ep_n - \bar{u}^\ep) \|_{L^1 (\mT^d)} \\
  & \hs{4mm} + \mE \| S (t^\ep_{n+1} - t^\ep_n) \bar{u}^\ep - \bar{u}^\ep \|_{L^1 (\mT^d)} + \mE \| \bar{u}^\ep - u^\ep_n \|_{L^1 (\mT^d)} \\
  & \leq \mE \| S(t^\ep_{n+1} - t^\ep_n) \bar{u}^\ep - \bar{u}^\ep \|_{L^1 (\mT^d)} + 2 \mE \| \bar{u}^\ep - u^\ep_n \|_{L^1 (\mT^d)} < \ep.
 \end{align*}Therefore by the definition of $t^\ep_{n+1}$, we must have $t^\ep_{n+1} = (t^\ep_{n} + \ep) \wedge T$. However, since $t^\ep_{n+1} < T$ by our assumption, it yields that $t^\ep_{n+1} = t^\ep_n + \ep$, which contradicts that $t^\ep_n \uparrow \bar{T}$
\end{proof}

\section{Convergence of the approximate solutions} 

We start from the following Chebyshev inequality:


\begin{lemm} \label{lemma5}
 Let $(X,\mu)$ be a finite mesure space and let $f \in L^1 (X)$. Set $M = \| f \|_{L^1 (X)}$. Then for any $\la > 0$, 
 \begin{gather*}
  \mu (\{ |f| \geq M \la^{-1} \}) \leq \la.
 \end{gather*}
\end{lemm}


\begin{proof} 
 Set $A = \{ |f| \geq M \la^{-1} \}$. Assume that $\mu (A) > \la$. Then we have
 \begin{gather*}
  M = \int_X | f | \ d\mu \geq \int_{A} | f | \ d\mu \geq \frac{M}{\la} \mu (A) > M.
 \end{gather*}This is a contradiction.
\end{proof}


\begin{prop}[Doubling of variables] \label{doubling}
Let $\ep , \ep' ,  \eta ,  \de > 0$. Then for all $t \in [0,T)$ we have
\begin{align*}
& - \mE \int_{(\mT^d)^2 \times \mR^2} \Big\{ f^+ (v^\ep(x,t), \xi) + f^+ ( \tilde{v}^\ep (x,t), \xi) - f^+ ( {v}^\ep (x,t^\ep), \xi) \Big\} \\
& \hs{4mm} \times \Big\{ f^- (v^{\ep'}(y,t), \zeta) + f^- ( \tilde{v}^{\ep'} (y,t), \zeta) - f^- ( {v}^{\ep'} (y,t^{\ep'}), \zeta) \Big\} \al_{\eta,\de} \ d\zeta d\xi dy dx \\
& \leq G (\eta, \de, \ep, \ep'), 
\end{align*}where $G (\eta, \de, \ep, \ep')$ is nonnegative and satisfies 
\begin{gather}
 \lim_{\eta \downarrow 0} \overline{\lim_{\ep, \ep' \downarrow 0}} G (\eta, \eta^{\theta}, \ep, \ep') = 0 \label{limit00}
\end{gather}for $\theta \in (1/\gamma,2)$.
\end{prop}


\begin{proof}
 To simplify the notation we will drop the variable $x$ of $v^\ep (x,t)$ and the variable $y$ of $v^{\ep'} (y,t)$. In the same way as in \cite[Proposition 3.2]{DeHoVo} and \cite[Proposition 3.2]{Ho}, we can obtain
 \begin{align*}
  & - \mE \int_{(\mT^d)^2 \times \mR^2} \Big\{ f^+ (v^\ep (t), \xi) + f^+ (\tilde{v}^\ep (t), \xi) - f^+ ({v}^\ep (t^\ep), \xi) \Big\} \\
  & \hs{7mm} \times \Big\{ f^- (v^{\ep'} (t), \zeta) + f^- (\tilde{v}^{\ep'} (t), \zeta) - f^- ({v}^{\ep'} (t^{\ep'}), \zeta) \Big\} \al_{\eta,\de} \ d\zeta d\xi dy dx \\
  & \leq - \mE \int_{(\mT^d)^2 \times \mR^2} f^+ (u_0, \xi) f^- (u_0, \zeta) \al_{\eta,\de} \ d\zeta d\xi dy dx \\
  & \hs{55mm} +I+J+K+I'+J_1'+J_2'+K'
\end{align*}where
\begin{align*}
  & I = - \mE \int_0^t \int_{(\mT^d)^2 \times \mR^2} f^+ (\tilde{v}^\ep (s), \xi) f^- (\tilde{v}^{\ep'} (s), \zeta) \\
  & \hs{40mm} \times (b(\xi) - b(\zeta)) \cdot \nabla_x \al_{\eta,\de} \ d\zeta d\xi dy dx ds,
  \end{align*}
  \begin{align*}
  & J = - \mE \int_0^t \int_{(\mT)^d \times \mR^2} f^+ (\tilde{v}^\ep (s), \xi) f^- (\tilde{v}^{\ep'} (s), \zeta) \\
  & \hs{40mm} \times (A(\xi)+A(\zeta)):D^2_x \al_{\eta,\de} \ d\zeta d\xi dy dx ds \\
  & \hs{6mm} - \mE \int_{[0,t] \times \mT^d \times \mR} \int_{\mT^d \times \mR} \al_{\eta,\de} \ d\de_{\tilde{v}^\ep (s)} (\xi) dx dn^{\ep '}_1 (y,s,\zeta) \\
  & \hs{6mm} - \mE \int_{[0,t] \times \mT^d \times \mR} \int_{\mT^d \times \mR} \al_{\eta,\de} \ d\de_{\tilde{v}^{\ep '} (s)} (\zeta) dy dn^\ep_1 (x,s,\xi),
  \end{align*}
  \begin{align*}
  & K = \frac{1}{2} \mE \int_0^t \int_{(\mT^d)^2 \times \mR^2} \sum_{k=1}^\infty \left| g_k (x,\xi) - g_k (y,\zeta) \right|^2 \al_{\eta,\de} \\
  & \hs{55mm} \times d\de_{v^{\ep'} (s)} (\zeta) d\de_{v^\ep (s)}(\xi) dy dx ds.
  \end{align*}
  \begin{align*}
   & I' = - \mE \int_0^t \int_{(\mT^d)^2 \times \mR^2} \Big\{ f^+ (v^\ep (s), \xi) - f^+ ({v}^\ep (s^\ep), \xi) \Big\} \\
   & \hs{40mm} \times f^- (\tilde{v}^{\ep'} (s), \zeta) b(\zeta) \cdot \nabla_y \al_{\eta,\de} \ d\zeta d\xi dy dx ds, \\
   & \hs{6mm} - \mE \int_0^t \int_{(\mT^d)^2 \times \mR^2} \Big\{ f^- (v^{\ep'} (s), \zeta) - f^- ({v}^{\ep'} (s^{\ep'}), \zeta) \Big\} \\
   & \hs{40mm} \times f^+ (\tilde{v}^\ep (s), \xi) b(\xi) \cdot \nabla_x \al_{\eta,\de} \ d\zeta d\xi dy dx ds,
  \end{align*}
  \begin{align*}
   & J_1' = - \mE \int_0^t \int_{(\mT^d)^2 \times \mR^2} \Big\{ f^+ (v^\ep (s), \xi) - f^+ ({v}^\ep (s^\ep), \xi) \Big\} \\
   & \hs{40mm} \times f^- (\tilde{v}^{\ep'} (s), \zeta) A(\zeta) : D^2_y \al_{\eta,\de} \ d\zeta d\xi dy dx ds, \\
   & \hs{6mm} - \mE \int_0^t \int_{(\mT^d)^2 \times \mR^2} \Big\{ f^- (v^{\ep'} (s), \zeta) - f^- ({v}^{\ep'} (s^{\ep'}), \zeta) \Big\} \\
   & \hs{40mm} \times f^+ (\tilde{v}^\ep (s), \xi) A(\xi) : D^2_x \al_{\eta,\de} \ d\zeta d\xi dy dx ds,
  \end{align*}
  \begin{align*}
   & J_2' = \mE \int_{[0,t] \times \mT^d \times \mR} \int_{\mT^d \times \mR} \Big\{ f^+ (v^\ep (s), \xi) - f^+ ({v}^\ep (s^\ep), \xi) \Big\} \\
   & \hs{65mm} \times \pa_\zeta \al_{\eta, \de} \ d\xi dx dm^{\ep'} (y,s,\zeta) \\
   & \hs{6mm} + \mE \int_{[0,t] \times \mT^d \times \mR} \int_{\mT^d \times \mR} \Big\{ f^- (v^{\ep'} (s), \zeta) - f^- ({v}^{\ep'} (s^{\ep'}), \zeta) \Big\} \\
   & \hs{65mm} \times \pa_\xi \al_{\eta, \de} \ d\zeta dy dm^{\ep} (x,s,\xi),
  \end{align*}and
  \begin{align*}
   & K' = - \frac{1}{2} \mE \int_0^t \int_{(\mT^d)^2 \times \mR^2} \Big\{ f^+ (\tilde{v}^\ep (s), \xi) - f^+ ({v}^\ep (s^\ep), \xi) \Big\} \\
   & \hs{50mm} \times G^2 (y,\zeta) \pa_\zeta \al_{\eta,\de} \ d\de_{v^{\ep'} (s) }(\zeta) d\xi dy dx ds \\
   & \hs{6mm} - \frac{1}{2} \mE \int_0^t \int_{(\mT^d)^2 \times \mR^2} \Big\{ f^- (\tilde{v}^{\ep'} (s), \zeta) - f^- ({v}^{\ep'} (s^{\ep'}), \zeta) \Big\} \\
   & \hs{50mm} \times G^2 (x,\xi) \pa_\xi \al_{\eta,\de} \ d\zeta d\de_{v^{\ep} (s) }(\xi) dy dx ds,
  \end{align*}where $n_1^\ep$ and $n_1^{\ep'}$ are the parabolic dissipation measures associated with $\tilde{v}^\ep$ and $\tilde{v}^{\ep '}$, respectively. By similar arguments as in \cite[Theorem 3.3]{DeHoVo} and \cite[Theorem 3.3]{Ho} we can obtain
  \begin{gather*}
   |I| \leq C \eta^{-1} \de, \ |J| \leq C \eta^{-2} \de^{2 \gamma}, \ K \leq C (\eta^2 \de^{-1} + r(\de)),
  \end{gather*}where a constant $C$ is independent of $\ep, \ep', \eta, \de$. We now show that
  \begin{gather*}
   \lim_{\ep,\ep' \downarrow 0} |I'| = \lim_{\ep,\ep' \downarrow 0} |J_1'| = \lim_{\ep,\ep' \downarrow 0} |J_2'| = \lim_{\ep,\ep' \downarrow 0} |K'| = 0.
  \end{gather*}We estimate the first term of $I'$ as follows:
  \begin{align*}
   & \Bigg| \mE \int_0^t \int_{(\mT^d)^2 \times \mR^2} \Big\{ f^+ (v^\ep (s), \xi) - f^+ ({v}^\ep (s^\ep), \xi) \Big\} \\
   & \hs{40mm} \times f^- (\tilde{v}^{\ep'} (s), \zeta) b(\zeta) \cdot \nabla_y \al_{\eta,\de} \ d\zeta d\xi dy dx ds \Bigg| \\
   & \leq \mE \int_0^t \int_{(\mT^d)^2 \times \mR_\zeta \times B_R} \Big| f^+ (v^\ep (s), \xi) - f^+ ({v}^\ep (s^\ep), \xi) \Big| |b(\zeta)| \big| \nabla_y \al_{\eta,\de} \big| \\
   & \hs{2mm} + \mE \int_0^t \int_{(\mT^d)^2 \times \mR_\zeta \times B_R^c} \Big| f^+ (v^\ep (s), \xi) - f^+ ({v}^\ep (s^\ep), \xi) \Big| |b(\zeta)| \big| \nabla_y \al_{\eta,\de} \big| \\
   & \leq C_{R,\eta,\de} \mE \int_0^t \int_{\mT^d} \left| v^\ep (s) - {v}^\ep (s^\ep) \right| dx ds \\
   & \hs{2mm} + C_{\eta,\de} \int_0^t \int_{\{ |v^\ep (s)| \wedge |\tilde{v}^\ep (s)| \geq R/2 \}} (1+|v^\ep (s)|^{p+1} + |\tilde{v}^\ep (s)|^{p+1}) \ d(\mathcal{L}_d \otimes P) ds,
  \end{align*}where $B_R = (-R,R)$, $p$ is the polynomial exponent of $b$ and $\mathcal{L}_d$ is the Lebesgue measure on $\mT^d$. It follows that 
  \begin{gather*}
   \mE \| v^\ep (s) - {v}^\ep (s^\ep) \|_{L^1 (\mT^d)} \to 0, \hs{3mm} \text{as} \hs{3mm} \ep \downarrow 0.
  \end{gather*}Indeed, we can get the above limit from the following calculation: if $s \in [ t^\ep_n, t^\ep_{n+1})$,
  \begin{align*}
   & \mE \| {v}^\ep (s) - u^\ep_n \|_{L^1 (\mT^d)} = \mE \int_{\mT^d} \left| \sum_{k=1}^\infty \int_{t_n^\ep}^s g_k (x, v (t)) \ d\beta_k (t) \right| dx \\
   & \hs{30.5mm} \leq C \mE \int_{\mT^d} \left( \int_{t_n^\ep}^{t_{n+1}^\ep} G^2 (x,v (t)) \ dt \right)^{\frac{1}{2}} dx \\
   & \hs{30.5mm} \leq C \left( \mE \int_{\mT^d} \int_{t_n^\ep}^{t_{n+1}^\ep} \left( 1+ \left| v (t) \right|^2 \right) \ dt dx \right)^{\frac{1}{2}} \\
   & \hs{30.5mm} \leq C \sqrt{\ep},
  \end{align*}where we used the Burkholder-Davis-Gundy inequality, (\ref{H3}) and Lemma \ref{lemma1}. Moreover it is easy to see from the estimates (\ref{apriori1}) and (\ref{apriori2}) that
  \begin{gather*}
   (\mathcal{L}_d \otimes P) \big( \{ | v^\ep (s) | \wedge | \tilde{v}^\ep (s) | > R  \} \big) \to 0, \hs{3mm} \text{as} \hs{3mm} R \to \infty \hs{3mm} \text{uniformly in} \hs{1.5mm} \ep.
  \end{gather*}
   Therefore by the Lebesgue convergence theorem we deduce $\lim_{\ep,\ep' \downarrow 0} |I'| = 0$. In the same manner as above, we also get that $\lim_{\ep, \ep' \downarrow 0} | J_1' | = 0$ (notice that $A(\xi)$ becomes bounded from the assumption (${\rm H_2}$)). Next we estimate $K'$ as follows. Set $D_R = \{ y \in \mT^d; |v^{\ep'} (s)| \leq R \}$. Then
   \begin{align*}
    & \bigg| \frac{1}{2} \mE \int_0^t \int_{(\mT^d)^2} \Big( \psi_\de (\tilde{v}^\ep (s) - v^{\ep'} (s)) - \psi_\de ({v}^\ep (s^\ep) - v^{\ep'} (s)) \Big) \\
    & \hs{62mm} \times G^2 (y, v^{\ep'} (s) ) \rho_\eta (x-y) \ dx dy ds \bigg| \\
    & \leq C \| \psi'_\de \|_{L^\infty (\mR)} \mE \int_0^t \int_{(\mT^d)^2} \Big( 1 + |v^{\ep'} (s)|^2 \Big) |\tilde{v}^\ep (s) - {v}^\ep (s^\ep)| \rho_\eta (x - y) \ dx dy ds \\
    & \leq C_\de \mE \int_0^t \int_{\mT^d} |\tilde{v}^\ep (s) - {v}^\ep (s^\ep)| \ dx ds \\
    & \hs{4mm} + C_\de \mE \int_0^t \int_{\mT^d} \int_{D_R} |\tilde{v}^\ep (s) - {v}^\ep (s^\ep)| R^2 \rho_\eta (x-y) \ dx dy ds \\
    & \hs{4mm} + C_{\de, \eta} \mE \int_0^t \int_{\mT^d} \int_{D_R^c} |\tilde{v}^\ep (s) - {v}^\ep (s^\ep)| |v^{\ep'} (s)|^2 \ dx dy ds \\
    & \leq C_\de (1 + R^2) T \ep \\
    & \hs{4mm} + C_{\de, \eta} \Big\{ \int_0^t \mE \int_{\mT^d} |\tilde{v}^\ep (s) - {v}^\ep (s^\ep)|^2 \ dxds \Big\}^{\frac{1}{2}} \Big\{ \int_0^t \mE \int_{D_R^c} | v^{\ep'} (s)|^4 \ dyds \Big\}^{\frac{1}{2}}.
   \end{align*}Here we used (\ref{H3}) and the following inequality
   \begin{gather*}
    \mE \int_{\mT^d} |\tilde{v}^\ep (s) - {v}^\ep (s^\ep)| \ dx \leq 2 \ep \hs{4mm} \text{for } s \in [0,T)
   \end{gather*}by our construction of the partition of $[0,T)$. Thanks to Lemma \ref{lemma1} we see that the second term on the right hand of the above inequality tends to $0$ as $R \to \infty$ uniformly in $\ep$ and $\ep'$. Consequently we obtain that $\lim_{\ep, \ep' \downarrow 0} |K'| = 0$. Finally the term $J_2'$ is somewhat subtle. To estimate it we set
   \begin{gather*}
    L_{\ep} = \sup_{0 < \ep' < 1} \mE \int_{\mT^d \times [0,T) \times \mR} \| \tilde{v}^\ep (s) - {v}^\ep (s^\ep) \|_{L^1 (\mT^d)} \ dm^{\ep'} (y,s,\zeta).
   \end{gather*}It is easy to see that
   \begin{gather*}
    |J'_2| \leq C_{\eta, \de} (L_\ep + L_{\ep'}).
   \end{gather*}By Lemma \ref{lemma2} and the arguments in \cite[Section 4.1.2]{DeVo2} there exist $m_\ep^0 \in L^2_w (\Om; \mathcal{M}_b)$ and a subsequence $\{ \ep'_n \}$, which might depend on $\ep$, such that $m^{\ep_n'} \rightharpoonup m^0_\ep$ in $L^2_w (\Om; \mathcal{M}_b)$-weak* and 
   \begin{gather*}
    L_{\ep} = \lim_{n \to \infty} \mE \int_{\mT^d \times [0,T) \times \mR} \| \tilde{v}^\ep (s) - {v}^\ep (s^\ep) \|_{L^1 (\mT^d)} \ dm^{\ep'_n} (y,s,\zeta),
   \end{gather*}where $\mathcal{M}_b = \mathcal{M}_b (\mT^d \times [0,T) \times \mR)$ is the space of finite measures on $\mT^d \times [0,T) \times \mR$ and $L^2_w (\Om; \mathcal{M}_b)$ is the space of all weak*-measurable mappings $m: \Om \to \mathcal{M}_b$ with $\mE \| m \|^2_{\mathcal{M}_b} < \infty$. Note that the sequence $\{ \ep_n ' \}$ doesn't necessarily converge to $0$. We define for a.s. the measure $m^0$ on $[0,T)$ by
   \begin{gather*}
    m^0 (A) = \sup_{0 < \ep < 1} \int_{\mT^d \times [0,T) \times  \mR} {\bf 1}_A (t) \ d m^0_\ep (x, t, \xi) \hs{4mm} \text{for } A \in \mathcal{B} ([0,T)).
   \end{gather*}In the standard terminology $m^0$ is probably called an outer measure, but we follow the book of Evans and Gariepy \cite{EvGa} in order to refer to it. Here we remark that $m^0$ can be considered as a weakly measurable map from $\Omega$ to the space of finite Radon (outer) measures on $[0,T)$, since we may assume without loss of generality that $\ep$ tends to $0$ through a sequence. Then we define the Radon measure $\nu$ on $[0,T)$ by
   \begin{gather*}
    \nu (A) = \mE m^0 (\cdot, A ) \hs{4mm} \text{for} \ A \in \mathcal{B}([0,T)).
   \end{gather*}We claim that $m^0 (\omega, \cdot)$ is absolutely continuous with respect to $\nu$ a.s. $\omega$. To show this we set
   \begin{gather*}
    \mathcal{N} = \{ B \in G_\de; \nu (B) = 0 \},
   \end{gather*}where $G_\de$ is the set of all $G_\de$-sets of $[0,T)$. For each $B \in \mathcal{N}$ there exists a null set $N_B \in \mathscr{F}$ satisfying $m^0 (\omega, B) = 0$ for all $\omega \in N_B^c$. Set $N = \cup_{B \in \mathcal{N}} N_B$. Since $\mathcal{N}$ is a countable set, $N$ is also null set of $\mathscr{F}$ such that
   \begin{gather*}
    m^0 (\omega, B) = 0 \hs{4mm} \text{for all} \ B \in \mathcal{N} \ \text{and all} \ \omega \in N^c.
   \end{gather*}Now take any $A \in \mathcal{B} ([0,T))$ with $\nu (A) = 0$. By the regularity of the Radon measure $\nu$ there exists $A^* \in G_\de$ satisfying $A \subset A^*$ and $\nu (A^*) = \nu (A) = 0$. Hence $A^* \in \mathcal{N}$ and $m^0 (\omega, A) = m^0 (\omega, A^*) = 0$ for all $\omega \in N^c$. Thus we have that $m^0 (\omega, \cdot) << \nu (\cdot)$ a.s. $\omega$. \\
\indent Let $D_\nu m^0$ be the derivative of $m^0$ with respect to $\nu$ (see \cite[p. 37]{EvGa}). Recall that $t \mapsto m^0 (\omega, (t-r, t+r))$ and $t \mapsto \nu ((t-r, t+r))$ are upper continuous by \cite[p. 39]{EvGa} and $\omega \mapsto m^0 (\omega, (t-r, t+r))$ is measurable thanks to $m^0_\ep \in L_w^2 (\Om; \mathcal{M}_b ([0,T)))$. Hence 
   \begin{gather*}
    (\omega, r) \mapsto \frac{m^0 (\omega, (t-r, t+r))}{\nu ((t-r, t+r))}
   \end{gather*}is $\mathscr{F} \otimes \mathcal{B} ([0,T))$-measurable. Then 
   \begin{gather*}
    D_\nu m^0 = \lim_{k \to \infty} \frac{m^0 (\omega, (t-1/k, t+ 1/k))}{\nu ((t - 1/k, t+ 1/k))} \hs{4mm} \nu \text{-a.e.}
   \end{gather*}and it is $\mathscr{F} \otimes \mathcal{B} ([0,T))$-measurable. By the Radon-Nikodym theorem (\cite[p. 40]{EvGa})
   \begin{gather*}
    m^0 (A) = \int_A D_\nu m^0 \ d\nu
   \end{gather*}for $P$-a.s., for $A \in \mathcal{B} ([0,T))$. We use Lemma \ref{lemma5} with $X = \Omega \times [0,T)$, $\mu = P \otimes \nu$ and $f = D_\nu m^0$. It is easy to compute that
   \begin{gather*}
    \nu (X) = \int_X |f| \ d\mu = \mE m^0 ([0,T)) =: M.
   \end{gather*}Set
   \begin{gather*}
    A_\la = \{ (\omega, s) \in \Om \times [0,T); |f| \geq \frac{M}{\la} \}, \hs{4mm} \la > 0.
   \end{gather*}Thus Lemma \ref{lemma5} assures that
   \begin{gather*}
    \int_{A_\la} \ d\nu dP = \mu (A_\la) \leq \la.
   \end{gather*}Then we write
   \begin{align*}
    & L_{\ep} \leq \mE \int_{[0,T)} \| \tilde{v}^\ep (s) - {v}^\ep (s^\ep) \|_{L^1 (\mT^d)} \ m^0 (\omega, ds) \\
    & \hs{4mm} = \mE \int_{[0,T)} {\bf 1}_{A_\la} (\omega, s) \| \tilde{v}^\ep (s) - {v}^\ep (s^\ep)  \|_{L^1 (\mT^d)} \ m^0 (\omega, ds) \\
    & \hs{7mm} + \mE \int_{[0,T)} {\bf 1}_{A^c_\la} (\omega, s) \| \tilde{v}^\ep (s) - {v}^\ep (s^\ep)  \|_{L^1 (\mT^d)} \ m^0 (\omega, ds) \\
    & \hs{4mm} =: L^1_\ep + L^2_\ep.
   \end{align*}By \cite[Section 1.3]{EvGa} and the construction of $\tilde{v}^\ep$, $L^2_\ep$ is estimated as follows.
   \begin{align*}
    & L^2_\ep = \mE \int_{[0,T)} {\bf 1}_{A^c_\la} (\omega, s) \| \tilde{v}^\ep (s) - {v}^\ep (s^\ep)  \|_{L^1 (\mT^d)} f (\omega, s) \ d\nu (s) \\
    & \hs{4mm} \leq \frac{M}{\la} \int_{[0,T)} \mE \| \tilde{v}^\ep (s) - {v}^\ep (s^\ep) \|_{L^1 (\mT^d)} \ d\nu (s) \\
    & \hs{4mm} \leq \frac{M}{\la} \ep \nu ([0,T)).
   \end{align*}On the other hand, in order to obtain an estimate of $L^1_\ep$, set
   \begin{gather*}
    A_0 = \lim_{\la \downarrow 0} A_\la = \bigcap_{\la > 0} A_\la.
   \end{gather*}Since $P \otimes \nu (A_0) \leq P \otimes \nu (A_\la) \leq \la$ for every $\la > 0$, it follows that $P \otimes \nu (A_0) = 0$. By the Fubini-Tonelli theorem we have $\nu ([A_0]_\omega) = 0$ for a.s. $\omega \in \Om$, where $[A_0]_\omega = \{ t \in [0,T); (\omega, t) \in A_0 \}$ is the $\omega$-section of $A_0$. Since $m^0 << \nu$ a.s., it follows that
   \begin{gather*}
    m^0 (\omega, [A_0]_\omega) = 0 \hs{4mm} \text{a.s. } \omega.
   \end{gather*}By the Schwartz inequality and Lemmas \ref{lemma1} and \ref{lemma2}, we have
   \begin{align*}
    & L^1_\ep \leq \Big\{ \mE \int_{[0,T)} \| \tilde{v}^\ep (s) - {v}^\ep (s^\ep) \|_{L^1 (\mT^d)}^2 \ dm^0 \Big\}^{\frac{1}{2}} \Big\{ \mE \int_{[0,T)} {\bf 1}_{A_\la} \ dm^0 \Big\}^{\frac{1}{2}} \\
    & \hs{4mm} \leq \Big\{ \mE \sup_s \| \tilde{v}^\ep (s) - {v}^\ep (s^\ep) \|_{L^1 (\mT^d)}^4 \Big\}^{\frac{1}{4}} \Big\{ \mE | m^0 ([0,T))|^2 \Big\}^{\frac{1}{4}} \Big\{ \mE m^0 (\omega, [A_\la]_\omega) \Big\}^{\frac{1}{2}} \\
    & \hs{4mm} \leq C \Big\{ \mE m^0 (\omega, [A_\la]_\omega) \Big\}^{\frac{1}{2}}.
   \end{align*}Consequently we have
   \begin{gather*}
    \lim_{\la \downarrow 0} \sup_{0 < \ep < 1} L^1_\ep \leq C \mE m^0 (\omega, [A_\la]_\omega) = 0.
   \end{gather*}Since 
   \begin{align*}
    \overline{\lim_{\ep \downarrow 0}} L_{\ep} \leq \sup_{0 < \ep < 1} L^1_\ep + \overline{\lim_{\ep \downarrow 0}} L^2_\ep = \sup_{0 < \ep < 1} L^1_\ep,
   \end{align*}letting $\la \downarrow 0$ gives $\lim_{\ep \downarrow 0} L_{\ep} = 0$. This means $\lim_{\ep, \ep' \downarrow 0}| J_2' | = 0$. Therefore if we let
   \begin{align*}
    & G (\eta, \de, \ep, \ep') =  - \mE \int_{(\mT^d)^2 \times \mR^2} f^+ (u_0, \xi) f^- (u_0, \zeta) \al_{\eta,\de} \ d\zeta d\xi dy dx \\
  & \hs{55mm} +I+J+K+I'+J_1'+J_2'+K',
   \end{align*}we get the assertion of the proposition.
\end{proof}


\begin{proof}[Proof of Theorem \ref{main}]Define
 \begin{align*}
  & F (\eta,\de,\ep,\ep') = \mE \int_{\mT^d} (v^\ep (x,t) - {v}^{\ep'} (x,t))^+ \ dx \\
  & \hs{19mm} + \mE \int_{(\mT^d)^2 \times \mR^2} \al_{\eta,\de} \Big\{ f^+ (v^\ep(t), \xi) + f^+ ( \tilde{v}^\ep (t), \xi) - f^+ ( {v}^\ep (t^\ep), \xi) \Big\} \\
  & \hs{25mm} \times \Big\{ f^- (v^{\ep'}(t), \zeta) + f^- ( \tilde{v}^{\ep'} (t), \zeta) - f^- ( {v}^{\ep'} (t^{\ep'}), \zeta) \Big\} \ d\zeta d\xi dy dx.
 \end{align*}Then by Proposition \ref{doubling} 
 \begin{gather}
  \mE \int_{\mT^d} (v^\ep (x,t) - {v}^{\ep'} (x,t))^+ \ dx \leq G (\eta,\de,\ep,\ep') + F (\eta,\de,\ep,\ep'). \label{cauchy}
 \end{gather}We now show that
 \begin{gather}
  \lim_{\eta \downarrow 0} \overline{\lim_{\ep, \ep' \downarrow 0}} F (\eta, \eta^\theta, \ep, \ep')=0. \label{limit}
 \end{gather}Observe that
 \begin{align*}
  & | F (\eta,\de,\ep,\ep')| \\
  & \leq \bigg| - \mE \int_{(\mT^d)^2 \times \mR^2} \Big\{ f^+ (v^\ep (t), \xi) + f^+ (\tilde{v}^\ep (t), \xi) - f^+ ({v}^\ep (t^\ep), \xi) \Big\} \\
  & \hs{17mm} \times \Big\{ f^- (v^{\ep'} (t), \zeta) + f^- (\tilde{v}^{\ep'} (t), \zeta) - f^- ({v}^{\ep'} (t^\ep), \zeta) \Big\} \al_{\eta, \de} \ d\zeta d\xi dy dx \\
  & \hs{7mm} + \mE \int_{\mT^d \times \mR} f^+ (v^\ep (t), \xi) f^- (v^{\ep'} (t), \xi) \ d\xi dx \bigg| \\
  & \leq \bigg| - \mE \int_{(\mT^d)^2 \times \mR^2} \Big\{ f^+ (v^\ep (t), \xi) + f^+ (\tilde{v}^\ep (t), \xi) - f^+ ({v}^\ep (t^\ep), \xi) \Big\} \\
  & \hs{17mm} \times \Big\{ f^- (v^{\ep'} (t), \zeta) + f^- (\tilde{v}^{\ep'} (t), \zeta) - f^- ({v}^{\ep'} (t^\ep), \zeta) \Big\} \al_{\eta, \de} \ d\zeta d\xi dy dx \\
  & \hs{7mm} + \mE \int_{(\mT^d)^2 \times \mR^2} f^+ (v^\ep (t), \xi) f^- (v^{\ep'} (t), \zeta) \al_{\eta,\de} \ d\zeta d\xi dy dx \bigg| \\
  & \hs{3mm} + \bigg| - \mE \int_{(\mT^d)^2 \times \mR^2} f^+ (v^\ep (t), \xi) f^- (v^{\ep'} (t), \zeta) \al_{\eta,\de} \ d\zeta d\xi dy dx \\
  & \hs{7mm} + \mE \int_{(\mT^d)^2 \times \mR} f^+ (v^\ep (x,t), \xi) f^- (v^{\ep'} (y,t), \xi) \rho_\eta (x-y) \ d\xi dy dx \bigg| \\
  & \hs{3mm} + \bigg| - \mE \int_{(\mT^d)^2 \times \mR} f^+ (v^\ep (x,t), \xi) f^- (v^{\ep'} (y,t), \xi) \rho_\eta (x-y) \ d\xi dy dx \\
  & \hs{7mm} + \mE \int_{\mT^d \times \mR} f^+ (v^\ep (t), \xi) f^- (v^{\ep'} (t), \xi) \ d\xi dx \bigg| \\
  & =: F_1 (\eta,\de,\ep,\ep') + F_2 (\eta,\de,\ep,\ep') + F_3 (\eta,\de,\ep,\ep').
 \end{align*}Since $\mE \int_{\mT^d} |\tilde{v}^\ep (x,t) - {v}^{\ep} (x,t^\ep)| \ dx < 2\ep$ by virtue of the construction of the approximate solutions, we easily get $F_1 (\eta,\de,\ep,\ep') < 2\ep + 2\ep'$. Moreover it is easy to see that $F_2 (\eta,\de,\ep,\ep') < \de$. Finally $F_3 (\eta,\de,\ep,\ep')$ is estimated as follows;
 \begin{align*}
  & \hs{-2mm} F_3 (\eta,\de,\ep,\ep') \\
  & = \bigg| - \mE \int_{\mT^d \times \mR} f^+ (v^\ep (x,t), \xi) \\
  & \hs{20mm} \times  \left\{ \int_{\mT^d} f^- (v^{\ep'} (y,t),\xi) \rho_\eta (x-y) \ dy - f^- (v^{\ep'} (x,t),\xi) \right\} \ d\xi dx \bigg| \\
  & \leq \mE \int_{(\mT^d)^2 \times \mR} \Big| f^- (v^{\ep'} (y,t), \xi) - f^- (v^{\ep'} (x,t), \xi) \Big| \rho_\eta (x-y) \ dydx \\
  & = 2 \mE \int_{(\mT^d)^2} \big( v^{\ep'} (y,t) - v^{\ep'} (x,t) \big)^+ \rho_\eta (x-y) \ dydx \\
  & \leq -2 \mE \int_{(\mT^d)^2} \int_{\mR^2} \Big\{ f^+ (v^{\ep'} (t), \zeta) + f^+ (\tilde{v}^{\ep'} (t), \zeta) - f^+ ({v}^{\ep'} (t^{\ep'}), \zeta) \Big\} \\
  & \hs{10mm} \times \Big\{ f^- (v^{\ep'} (t), \xi) + f^- (\tilde{v}^{\ep'} (t), \xi) - f^- ({v}^{\ep'} (t^{\ep'}), \xi) \Big\} \al_{\eta,\de} \ d\zeta d\xi dy dx \\
  & \hs{4mm} + F_1 (\eta,\de,\ep',\ep') + F_2 (\eta,\de,\ep',\ep') \\
  & \leq G (\eta,\de,\ep',\ep') + 4 \ep' + 2 \de,
 \end{align*}where $G$ is the same function that appears in Proposition \ref{doubling}. Thus we obtain the limit (\ref{limit}). Consequently we have that $\{ v^\ep; \ep>0 \}$ is a Cauchy sequence in $L^\infty (0,T; L^1 (\Omega \times \mT^d))$ from (\ref{limit00}), (\ref{cauchy}) and (\ref{limit}). Besides, by (\ref{timeL101}) and (\ref{timeL102}) we have
 \begin{align*}
  & \hs{-2mm} \mE \| v^\ep (t) - \tilde{v}^\ep (t) \|_{L^1 (\mT^d)} \\
  & \hs{2mm} \leq \mE \| v^\ep (t) - {v}^\ep (t^\ep) \|_{L^1 (\mT^d)} + \mE \| {v}^\ep (t^\ep) - \tilde{v}^\ep (t) \|_{L^1 (\mT^d)} \leq C \ep^{1/2} + \ep
 \end{align*}for all $t \in [0,T)$. Therefore, $\{ \tilde{v}^\ep; \ep > 0 \}$ is also a Cauchy sequence and it's limit is the same as the limit of $\{ v^\ep; \ep > 0 \}$. 
 

Once one has obtained that the approximate solution $\{ v^\ep \}$ ( or $\{ \tilde{v}^\ep \}$ ) converges to $u$ in the sense of $L^\infty (0,T;L^1 (\Omega \times \mT^d))$-norm, one can proceed to the same arguments as in \cite[Theorem 6.4]{DeHoVo}. In particular, $\{ v^\ep \}$ ( or $\{ v^{\ep'} \}$ ) is a Cauchy sequence in $L^1 (\Om \times (0,T), \mathcal{P}, dP \otimes dt; L^1 (\mT^d))$, and hence the limit $u$ is also predictable. From Lemmas \ref{lemma1} and \ref{lemma2} there exist kinetic measures $m$, $o_1$ and a positive null sequence $\{ \ep_n \}$ such that
 \begin{align*}
  & m^{\ep_n} \rightharpoonup m \hs{3mm} \text{in} \hs{3mm} L_w^2 (\Omega; \mathcal{M}_b)\text{-weak*}, \\
  & n_1^{\ep_n} \rightharpoonup o_1 \hs{3mm} \text{in} \hs{3mm} L_w^2 (\Omega; \mathcal{M}_b)\text{-weak*}, \\
  & \text{and} \hs{2mm} m \geq o_1 \geq n_1 \hs{2mm} \text{a.s.},
 \end{align*}where $n_1$ is defined by (\ref{measure}) with the function $u$ (for a detailed exposition we refer the reader to \cite[Theorem 6.4]{DeHoVo}). Let $\ph \in C_c^\infty (\mT^d \times \mR)$. By the same way as in the proof of \cite[Theorem 6.4]{DeHoVo}, we have
 \begin{gather*}
  \int_{\mT^d} \int_\mR f^+ (v^{\ep_n}, \xi) \ph (x,\xi) \ d\xi dx \to \int_{\mT^d} \int_\mR f^+ (u, \xi) \ph (x,\xi) \ d\xi dx \hs{4mm} \text{a.e. } \omega, t
 \end{gather*}
  Using the It\^{o} isometry and the dominated convergence theorem, we have
  \begin{gather*}
  \hs{-30mm} \sum_{k=1}^\infty \int_0^t \int_{\mT^d} g_k (x,v^{\ep_n}) \ph (x,v^{\ep_n}) \ dx d\beta_k (s) \nn \\
  \hs{24mm} \to \sum_{k=1}^\infty \int_0^t \int_{\mT^d} g_k (x,u) \ph (x,u) \ dx d\beta_k (s), \hs{4mm} \text{in} \hs{2mm} L^2(\Omega)
 \end{gather*}by selecting a subsequence if necessary. Hence by selecting a further subsequence if necessary, the just above limit holds almost surely. Finally, using the dominated convergence theorem again, we have
 \begin{gather*}
  \hs{-30mm} \int_0^t \int_{\mT^d} G^2 (x,v^{\ep_n}) \pa_\xi \ph (x,v^{\ep_n}) \ dxds \\
  \hs{24mm} \to \int_0^t \int_{\mT^d} G^2 (x,u) \pa_\xi \ph (x,u) \ dxds, \hs{4mm} \text{a.s.}
 \end{gather*}Therefore passing to the limit in (\ref{kine}), we have
 \begin{align*}
& - \int_{\mT^d} \int_\mR f^\pm (u (t), \xi) \ph \ d\xi dx + \int_{\mT^d} \int_\mR f^\pm (u_0, \xi) \ph \ d\xi dx \nn \\
& \hs{7mm} + \int_0^t \int_{\mT^d} \int_{\mR} f^\pm (u (s), \xi) (b (\xi) \cdot \nabla + A (\xi):D^2 ) \ph \ d\xi dx ds \nn \\
& = - \sum_{k=1}^\infty \int_0^t \int_{\mT^d} g_k (x,u (s)) \ph (x,u(s)) \ dx d\beta_k (s) \nn \\
& \hs{7mm} -\frac{1}{2} \int_0^t \int_{\mT^d} G^2 (x,u(s)) \pa_\xi \ph (x,u(s)) \ dx ds + \int_{[0,t] \times \mT^d \times \mR} \pa_\xi \ph \ dm , 
\end{align*}for a.e. $\omega, t$. Multiplying the above by $\psi' (t)$, $\psi \in C_c^\infty ([0,T))$, and integrating with respect to $t \in [0,T)$, we can see that $u$ satisfies the kinetic formulation (\ref{kine1}).  Therefore we conclude that $u$ is a kinetic solution to (\ref{dp1}), (\ref{dp2}).
 

Moreover, the $L^1$-contraction property is a straightforward consequence of a comparison result which follows from Proposition \ref{doubling}, and the fact that a kinetic solution has almost surely continuous trajectories is inferred from \cite[Corollary 3.4]{DeHoVo}. Thus the proof is complete. \\
\end{proof}

\noindent {\bf Acknowledgements} \mbox{} \\[-3mm]

The first author have been partialy supported by Grant-in-Aid Sientific Research the Japan Society for the Promotion
of Science and the second author by Waseda University Grant for Special Research Projects (No. 20155-042). The authors wish to thank the referees very heartily for their comments helping to improve the manuscript.



\bibliographystyle{model1-num-names}
\bibliography{<your-bib-database>}



\end{document}